\documentclass[11pt]{article}

\usepackage[utf8]{inputenc}
\usepackage[english]{babel}
\usepackage{enumerate}
\usepackage{latexsym}
\usepackage{amsthm,amsmath}
\usepackage{amssymb}
\usepackage{aliascnt} %para que funcione bien \Cref
\usepackage{multicol}
\usepackage{tikz}
\usepackage{float} % for [H] option in figures

\usepackage[bookmarks]{hyperref} 
\hypersetup{
	pdftitle={2-switch-degree Classification of Split Graphs},
	pdfauthor={V.N. Schvollner},
	colorlinks=true,
	linkcolor=blue,
	citecolor=red,
	filecolor=cyan,
	urlcolor=magenta
}

\usetikzlibrary{babel,decorations.pathreplacing,calc,positioning,shapes.geometric,arrows.meta}

\usepackage{cleveref}

\definecolor{10}{RGB}{115,59,171}
\definecolor{8}{RGB}{212,122,240}
\definecolor{7}{RGB}{99,212,119}
\definecolor{6}{RGB}{183,240,164}
\definecolor{D}{RGB}{255,162,79}
\definecolor{E}{RGB}{255,84,0}
\definecolor{F}{RGB}{158,248,255}
\definecolor{G}{RGB}{128,135,255}
\definecolor{I}{RGB}{187,255,0}
\definecolor{A}{cmyk}{.9,.05,.4,0}
\definecolor{B}{RGB}{150,30,150}
\definecolor{C}{RGB}{186,155,189}
\definecolor{9}{RGB}{0,180,60}
\definecolor{0}{RGB}{30,123,191}
\definecolor{1}{RGB}{255,113,102}
\definecolor{2}{RGB}{41,199,92}
\definecolor{3}{RGB}{242,207,16}
\definecolor{5}{RGB}{255,15,154}
\definecolor{4}{rgb}{.8,0,.8}

\definecolor{Red}{rgb}{1,0.4,0.4}
\definecolor{Green}{rgb}{.1,.5,.1}
\definecolor{Blue}{rgb}{.1,.1,.5}
\definecolor{blue}{RGB}{0,0,255}
\definecolor{Yellow}{rgb}{.8,.4,0}
\definecolor{X}{rgb}{.8,.4,0}
\definecolor{H}{rgb}{0,0,1}
\definecolor{light}{rgb}{.67,.84,.90}
\definecolor{Cyan}{rgb}{0,1,1}
\definecolor{Purple}{rgb}{.5,0,.5}
\definecolor{Purple2}{rgb}{.5,.2,.5}
\definecolor{white}{rgb}{1.0,1.0,1.0}
\definecolor{Purple2}{rgb}{.8,.4,0}
\definecolor{Amarillo}{RGB}{225,191,73}
\definecolor{Celeste}{RGB}{117,170,219}
\definecolor{Castano}{RGB}{232,53,17}
\definecolor{Black}{RGB}{0,0,0}
\definecolor{White}{RGB}{255,255,255}
\definecolor{gris}{rgb}{.5,.5,.5}

\newtheorem{theorem}{Theorem}[section]
\newtheorem{conjecture}[theorem]{Conjecture}

\newaliascnt{corollary}{theorem}
\newtheorem{corollary}[corollary]{Corollary}
\aliascntresetthe{corollary}
\newaliascnt{lemma}{theorem}
\newtheorem{lemma}[lemma]{Lemma}
\aliascntresetthe{lemma}
\newtheorem{problem}[theorem]{Problem}
\newaliascnt{proposition}{theorem}
\newtheorem{proposition}[proposition]{Proposition}
\aliascntresetthe{proposition}
\theoremstyle{definition}

\pgfdeclarelayer{background2}
\pgfdeclarelayer{background}
\pgfdeclarelayer{foreground}
\pgfsetlayers{background2,background,main,foreground}

\title{The $\Delta$ property: a bridge between split graphs\\ and Number Theory}
\author{Victor N. Schvöllner}
\date{}

\begin{document}
	
	\maketitle
	\begin{abstract}
			%A \emph{2-switch} is a local edge swap that preserves the degree sequence of a graph. 
			For a split graph $S$, the combinatorics of 2-switches on $S$ is faithfully encoded by the factor graph $\Phi(S)$, a multigraph whose induced cycles have length at most $4$. In this paper we address the following question: for which $n \in \mathbb{N}$ is there a split graph $S$ whose factor graph contains an $n$-simple triangle, that is, a triangle all of whose edges have multiplicity $n$? We show that the answer is governed by a purely arithmetic condition, the $\Delta$ property, relating the differences and sums of complementary divisors of $n$, and thereby establish a two-way bridge between Graph Theory and Number Theory. 
			%We develop the theory of the set $\mathbb{N}(\Delta)$ of integers satisfying this condition: we introduce its multiplicative atoms (the \emph{$\Delta$-primitives}), prove their infinitude via an elementary cubic polynomial generator, exhibit several explicit infinite families of obstructions, and characterize exactly when a squarefree product of three primes lies in $\mathbb{N}(\Delta)$.
%		We address a problem that lies at the intersection of Graph Theory and Number Theory: establishing conditions on $n\in\mathbb{N}$ for the existence of a split graph $S$ such that $\vec{\Phi}(S)$ is an $n$-simple triangle of type 0 (see  \Cref{triangulos.permitidos}). This problem may initially seem overly specific, but in fact, considering that the cycles in $\Phi$ can only be of size 3 or 4, it represents a significant step forward in understanding the 2-switch structure of split graphs. Moreover, the bridge it creates with Number Theory opens many doors to future research in this area.
	\end{abstract}
	


\section{Introduction}
\label{sec:intro}

\subsection{2-switches, realization graphs, and indecomposable split graphs}

Let $G$ be a graph. By $V(G)$ and $E(G)$ we denote, respectively, the set of vertices of $G$ and the set of edges of $G$. Given four distinct vertices $a,b,x,y$ of $G$ with $ab, xy \in E(G)$ and $ax, by \notin E(G)$, the process of replacing $ab, xy$ with $ax, by$ is said to be a \emph{2-switch} on $G$. This local operation preserves the degree sequence of $G$, and a classical theorem asserts that any two graphs with the same degree sequence are connected by a sequence of 2-switches (see \cite{chartrand2010graphs}). The \emph{realization graph} $\mathcal{G}(d)$ of a degree sequence $d$, whose vertices are the graphs with degree sequence $d$ and whose edges correspond to 2-switches, is therefore connected; understanding its geometry has been an active line of research for several decades (see, e.g.,~\cite{arikati1999realization,vnsigma.2switch.unic.pseudof}).

A particularly natural invariant extracted from $\mathcal{G}(d)$ is the \emph{2-switch-degree} $\deg(G)$ of a graph $G$, namely the degree of $G$ as a vertex of $\mathcal{G}(d)$. Equivalently, $\deg(G)$ counts the number of distinct 2-switches acting on $G$. This parameter, first studied systematically in~\cite{vnsigma.2switch.degree}, measures the local flexibility of $G$ within its realization class and is closely related to several structural properties.

A graph is said to be \textit{split} if its vertex set can be partitioned into a clique and an independent set (see \cite{cheng2016split}). Two results situate split graphs at the heart of the 2-switch-degree theory. On one hand, Tyshkevich proved in \cite{tyshkevich2000decomposition} that every graph $G$ admits a unique decomposition $G = G_r\circ \cdots \circ G_1$ into indecomposable factors, with the property that $G_2, \ldots, G_r$ are all split graphs. On the other hand, Barrus and West introduced in \cite{barrus.west.A4} the auxiliary graph $A_4(G)$, whose vertices are those of $G$ and whose edges join pairs of vertices participating together in some 2-switch of $G$, and showed that $G$ is indecomposable in the sense of Tyshkevich if and only if $A_4(G)$ is connected. Together with the fact that $\deg(S\circ G)=\deg(S)+\deg(G)$ (see \cite{vnsigma.2switch.degree}), these two theorems reduce the problem of classifying all the split graphs with fixed degree to the problem of classifying all indecomposable split graphs with fixed degree.

\subsection{The factor graph of a split graph}

Motivated by this reduction, in~\cite{vnsigma.factor.graph} the authors introduced the \emph{factor graph} $\Phi(S)$ of a split graph $S$. Let $(K,I)$ be a bipartition of $V(S)$ into a clique $K$ and an independent set $I$. So, $\Phi(S)$ is the loopless multigraph on $I$ in which the multiplicity $\sigma_{uv}$ of the edge $uv$ records the number of 2-switches acting simultaneously on $u$ and $v$. Throughout this paper, the symbols $d_v$ and $N_v$ always denote, respectively, the degree and the neighborhood of a vertex $v$ in the split graph $S$. An explicit formula relates $\sigma_{uv}$ to the neighborhoods of $u$ and $v$ in $S$:
\[
\sigma_{uv} \;=\; \bigl(d_u - \eta_{uv}\bigr)\bigl(d_v - \eta_{uv}\bigr),
\]
where $\eta_{uv} = |N_u \cap N_v|$. Summing multiplicities gives $|E(\Phi(S))| = \deg(S)$, so the factor graph packages the 2-switch-degree of $S$ together with its fine combinatorial structure.
%Compared with $A_4(S)$, the factor graph $\Phi(S)$ is cleaner, more compact, and free of the redundancies that $A_4(S)$ exhibits on split graphs. 
For this reason, $\Phi(S)$ has emerged as the natural object to study the 2-switch dynamics of indecomposable split graphs (see \cite{vnsigma.factor.graph, vnsigma.factor.graph.paths}).

The \textit{flow configuration} of a split graph $(S,K,I)$, denoted by $\vec{\Phi}(S)$, is defined as the digraph with $I$ as vertex set, where there is an arc $(u,v)$ from $u$ to $v$ if and only if $d_u\leq d_v$ in $S$ and $\sigma_{uv}>0$. Clearly, $\vec{\Phi}(S)$ and $\Phi(S)$ are isomorphic as simple graphs, ignoring multiple edges and arc directions. 
%Furthermore: $uv\in\vec{\Phi}(S)$ and $vu\notin\vec{\Phi}(S)$ if and only if $\sigma_{uv}(S)>0$ and $\deg_S(u)<\deg_S(v)$.
%Orienting each edge of $\Phi(S)$ from the endpoint of smaller degree (in $S$) to the endpoint of larger degree (in S) yields the \emph{oriented factor graph} $\vec\Phi(S)$; 
A key structural fact, established in~\cite{vnsigma.factor.graph.paths}, is that every induced cycle in $\Phi(S)$ has length at most $4$. It is easy to see that triangles of $\vec{\Phi}$ fall into one of the four orientation types of \Cref{triangulos.permitidos}. In this article we focus on the simplest and most symmetric case: triangles in which all three edges share a common multiplicity.

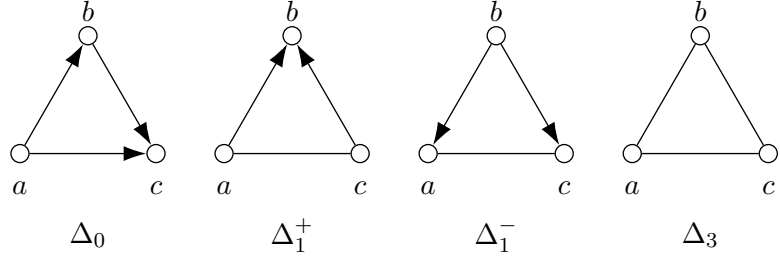
\begin{figure}[H]
	\centering
	\begin{tikzpicture}[
		scale=0.9,
		every node/.style={circle, draw, inner sep=2.4 pt},
		line width=0.5pt
		]
		
		% Primer triangulo: Delta_0
		\begin{scope}[shift={(0,0)}]
			\node (a) at (0,0) {};
			\node (c) at (2,0) {};
			\node (b) at (1,1.732) {};
			
			\draw[-{Latex[length=3mm,width=2mm]}] (a) -- (b);
			\draw[-{Latex[length=3mm,width=2mm]}] (b) -- (c);
			\draw[-{Latex[length=3mm,width=2mm]}] (a) -- (c);
			
			\node[draw=none] at (0,-0.5) {$a$};
			\node[draw=none] at (2,-0.5) {$c$};
			\node[draw=none] at (1,2.1) {$b$};
			
			\node[draw=none] at (1,-1.2) {$\Delta_0$};
		\end{scope}
		
		% Segundo triangulo: Delta_1^+
		\begin{scope}[shift={(3,0)}]
			\node (a) at (0,0) {};
			\node (c) at (2,0) {};
			\node (b) at (1,1.732) {};
			
			\draw[-{Latex[length=3mm,width=2mm]}] (a) -- (b);
			\draw[-{Latex[length=3mm,width=2mm]}] (c) -- (b);
			\draw (a) -- (c);
			
			\node[draw=none] at (0,-0.5) {$a$};
			\node[draw=none] at (2,-0.5) {$c$};
			\node[draw=none] at (1,2.1) {$b$};
			
			\node[draw=none] at (1,-1.2) {$\Delta_1^+$};
		\end{scope}
		
		% Tercer triangulo: Delta_1^-
		\begin{scope}[shift={(6,0)}]
			\node (a) at (0,0) {};
			\node (c) at (2,0) {};
			\node (b) at (1,1.732) {};
			
			\draw[-{Latex[length=3mm,width=2mm]}] (b) -- (a);
			\draw[-{Latex[length=3mm,width=2mm]}] (b) -- (c);
			\draw (a) -- (c);
			
			\node[draw=none] at (0,-0.5) {$a$};
			\node[draw=none] at (2,-0.5) {$c$};
			\node[draw=none] at (1,2.1) {$b$};
			
			\node[draw=none] at (1,-1.2) {$\Delta_1^-$};
		\end{scope}
		
		% Cuarto triangulo: Delta_3
		\begin{scope}[shift={(9,0)}]
			\node (a) at (0,0) {};
			\node (c) at (2,0) {};
			\node (b) at (1,1.732) {};
			
			\draw (a) -- (b);
			\draw (b) -- (c);
			\draw (a) -- (c);
			
			\node[draw=none] at (0,-0.5) {$a$};
			\node[draw=none] at (2,-0.5) {$c$};
			\node[draw=none] at (1,2.1) {$b$};
			
			\node[draw=none] at (1,-1.2) {$\Delta_3$};
		\end{scope}
		
	\end{tikzpicture}
	\caption{Permitted triangles in $\vec{\Phi}$. The sub-index denotes the number of edges of the digraph.}
	\label{triangulos.permitidos}
\end{figure}
%\begin{figure}[H]
%	\centering
%	\includegraphics[scale=0.8]{triangulos_permitidos.pdf}
%	\caption{Permitted triangles in $\vec{\Phi}$.}
%	\label{triangulos.permitidos}
%\end{figure}
 If $H$ is a subgraph of $\Phi$, we denote by $\vec{H}$ the corresponding subgraph in $\vec{\Phi}$. We say that $H$ is $n$-simple if all the edges of $H$ have multiplicity $n$. A triangle $T$ in $\Phi$ is said to be of \textit{type 0} if $\vec{T}=\Delta_0$ (up to labeling).
%A split graph $S$ is called  $n$\textit{-simple} if $\Phi(S)$ is $n$-simple. 
The question driving this paper can now be stated cleanly.

\begin{problem}\label{prob:main}
	For which natural numbers $n$ does there exist a split graph $S$ such that $\Phi(S)$ is an $n$-simple triangle of type 0?
\end{problem}

\subsection{The bridge to Number Theory: the $\Delta$ property}

At first glance, Problem~\ref{prob:main} may appear overly specific. Its interest lies in an unexpected phenomenon: the set of natural numbers for which the answer to Problem \ref{prob:main} is affirmative admits a clean, purely arithmetic description. To state it, for $n \in \mathbb{N}=\{1,2,3,\ldots\}$ let $D_n$ denote the set of positive divisors of $n$, and define
\[
D^*_n \;=\; \bigl\{\, |a - b| : a, b \in D_n,\ ab = n \,\bigr\}, \quad
D^+_n \;=\; \bigl\{\, x + y : x, y \in D^*_n - \{0\} \,\bigr\}.
\]
That is, $D^*_n$ collects the differences between complementary divisors of $n$, and $D^+_n$ the pairwise sums of the nonzero such differences. We say that $n$ has the \emph{$\Delta$ property} (or satisfies the $\Delta$ condition) when
\[
D^*_n \cap D^+_n \;\neq\; \varnothing,
\]
and we denote by $\mathbb{N}(\Delta)$ the set of all natural numbers with the $\Delta$ property. By inspection, the two smallest members of $\mathbb{N}(\Delta)$ are $24$ and $40$, because
\[
\tfrac{24}{2} - 2 \;=\; 10 \;=\; 2\left(\tfrac{24}{3} - 3\right), \qquad \tfrac{40}{4} - 4 \;=\; 6 \;=\; 2\left(\tfrac{40}{5} - 5\right).
\]
%witnessed respectively by the divisor triples $(2,3,3)$ and $(4,5,5)$; 
This article is divided into 5 sections. The first part of \Cref{sec:Graphs<->NumberTheory} is devoted to basic facts about the $\Delta$ condition, encoding them through the notion of a \emph{$\Delta$-triple} $(x,y,z)$ of divisors of $n$. We bound the components of any such triple, and determine the optimal upper bound on $n$ in terms of its smallest component $x$, separately in the regimes $y=z$ and $y<z$ (\Cref{prop:max-duplicated,prop:max-generic}). One of the main results of this work then links Problem \ref{prob:main} to the $\Delta$ condition in the second (last) part of \Cref{sec:Graphs<->NumberTheory}:

\begin{theorem}[\Cref{triang.n-simple&tipo0.implica.Delta.prop,n.con.propDelta.implica.existencia.de.T_0.n-simple}, \Cref{cor_finite_number_active_split}]
	\label{thmintro:A}
	Let $n \in \mathbb{N}$. Then, $n \in \mathbb{N}(\Delta)$ if and only if there exists an indecomposable split graph $S$ such that $\Phi(S)$ is $n$-simple triangle of type 0 (see \Cref{triangulos.permitidos}).
	%Moreover, whenever $n \in \mathbb{N}(\Delta)$, there are infinitely many pairwise non-isomorphic balanced split graphs $S$ realizing (a), only finitely many of which are active.
\end{theorem}

The forward implication in \Cref{thmintro:A} solves an explicit realization problem starting from a number-theoretic datum, while the reverse implication extracts an arithmetic invariant from a purely combinatorial configuration. Together they provide a two-way bridge between Graph Theory and Number Theory. 

\subsection{The set $\mathbb{N}(\Delta)$ as a number-theoretic object}

\Cref{thmintro:A} motivates the systematic study of $\mathbb{N}(\Delta)$ as a number-theoretic object, which occupies the remainder of the article, i.e., from \Cref{sec:Delta.prim} and beyond. 

%\paragraph{Multiplicative structure.}
In \Cref{sec:Delta.prim} we observe that $\alpha^{2} n \in \mathbb{N}(\Delta)$ whenever $n \in \mathbb{N}(\Delta)$ and $\alpha \in \mathbb{N}$, which immediately yields $|\mathbb{N}(\Delta)| = \infty$ and motivates the notion of a \emph{$\Delta$-primitive} number: an element of $\mathbb{N}(\Delta)$ that is not a nontrivial square multiple of any smaller element of $\mathbb{N}(\Delta)$. $\Delta$-primitives play the role of multiplicative atoms for $\mathbb{N}(\Delta)$, in loose analogy with primes for $\mathbb{N}$: every member of $\mathbb{N}(\Delta)$ decomposes as the square of an integer times a $\Delta$-primitive, and every square-free element of $\mathbb{N}(\Delta)$ is itself $\Delta$-primitive. We further show that the subset $\mathbb{N}(\Delta^{2}) \subset \mathbb{N}(\Delta)$ of squares with the $\Delta$ property is closed under multiplication; it therefore forms an abelian semigroup under the ordinary product. We then study the uniqueness of the decomposition $n = \alpha^{2} m$ into a square and a $\Delta$-primitive: we prove that it is unique whenever two such decompositions have coprime $\Delta$-primitive parts (\Cref{prop:coprime-uniqueness}), but that uniqueness fails in general. This leads to a refined uniqueness conjecture, phrased in terms of the $\Delta$-primitives sharing a given square-free part. Finally, we close the section by constructing a cubic polynomial generator of numbers with the $\Delta$ property, through which we prove that there are infinitely many $\Delta$-primitives (\Cref{Delta-primitivos.son.infinitos}).

In \Cref{sec:polinom.generadores} we develop a systematic way to produce elements of $\mathbb{N}(\Delta)$ through generating polynomials. The starting point is that
\[
x(2x-1)(3x-2) \in \mathbb{N}(\Delta) \qquad \text{for every integer } x \geq 2,
\]
which is used to prove the infinitude of $\Delta$-primitive numbers in \Cref{subsec:delta_prim_are_infinite}.
%which encodes a solution of the $\Delta$ condition through the divisor triple $(x, 2x-1, 2x-1)$. 
We then generalize the construction and exhibit an infinite family of cubic polynomials $n(x)$ whose values lie in $\mathbb{N}(\Delta)$ for all sufficiently large integers $x$ (\Cref{polinom.Delta-generadores_deg=3}). Two complementary results delimit the reach of this method: it produces no generating polynomial of degree $\geq 4$ (\Cref{thm:failure-kgeq4}) and no linear family realizing the regime $y < z$ in degree $3$ (\Cref{thm:failure-generic-k3}).

 Whether $\mathbb{N}(\Delta^{2})$ contains infinitely many $\Delta$-primitive squares remains an open problem (Conjecture \ref{conj:inf_many_delta-squares}). However, we note a suggestive connection: for each cubic generating polynomial $f$ produced by our method, locating squares in $f(\mathbb{N}) \cap \mathbb{N}(\Delta)$ amounts to finding integer points on the elliptic curve $y^{2} = f(x)$. By Siegel's theorem (see \cite{siegel1929}), each such curve contributes only finitely many integer points; however, the infinite family of generating polynomials given by \Cref{polinom.Delta-generadores_deg=3} raises the possibility of constructing a family of elliptic curves that collectively yields infinitely many $\Delta$-primitive squares. We develop this perspective in \Cref{subsec:elliptic}.

%\paragraph{An elementary polynomial method.}

%evaluating the generating polynomial $f(x) = x(2x-1)(3x-2)$ at a suitably chosen prime $p$ forces $f(p)$ to be $\Delta$-primitive via a short argument based on divisibility by $p$.

%\paragraph{Obstructions: integers without the $\Delta$ property.}
The complementary part of the theory concerns integers that \emph{fail} the $\Delta$ condition. This is the content of \Cref{sec:Num.sin.peop.Delta}. Two guiding heuristics emerge from our analysis. The first is that a prime appearing in the factorization of $n$ which is ``large'' compared to $n$ obstructs the $\Delta$ property: concretely, $pk \notin \mathbb{N}(\Delta)$ whenever $p \geq k$, and an analogous phenomenon occurs for numbers of the form $p^{x}q^{y}$ when $q > p^{x}$, where $p,q$ are primes (\Cref{p>2k-2_entonces.pk.no tiene.prop.Delta,p^xq^y.no.tiene.prop.Delta.q>p^x}). The second heuristic is that integers with few distinct prime factors tend to fall outside $\mathbb{N}(\Delta)$: this is true for all prime powers, as well as for several families of integers with exactly two prime divisors. Sharper than these heuristics is a complete characterization of which products of three distinct primes lie in $\mathbb{N}(\Delta)$ (\Cref{pqr.prop.Delta}). As a consequence, the obstructions of \Cref{sec:Num.sin.peop.Delta} restrict the possible values of the divisor-counting function $\tau$ on $\mathbb{N}(\Delta)$ (\Cref{cor:tau-cota-refined}).

%Concretely, we prove that:
%\begin{itemize}
%	\item $2k \notin \mathbb{N}(\Delta)$ whenever $k$ is odd, so $\mathbb{N} - \mathbb{N}(\Delta)$ is infinite and every even element of $\mathbb{N}(\Delta)$ is a multiple of $4$;
%	\item $p^{k} \notin \mathbb{N}(\Delta)$ for every prime $p$ and every $k \geq 1$;
%	\item $pk \notin \mathbb{N}(\Delta)$ for every $k \geq 1$ and every prime $p \geq 2k-1$ (a quantitative form of the first heuristic);
%	\item $p^{x}q^{y} \notin \mathbb{N}(\Delta)$ whenever $p, q$ are distinct primes with $q > p^{x}$ and $y \geq 2$;
%	\item $\{pq,\, p^{2}q,\, pq^{2},\, p^{2}q^{2}\} \cap \mathbb{N}(\Delta) = \varnothing$ for all primes $p, q$ (a quantitative form of the second heuristic);
%	\item if $p^{k} q \in \mathbb{N}(\Delta)$, then $p^{k} q \in \{\, 2^{2h+1} q : q \in \{3, 5\},\ h \in \mathbb{N}\,\}$;
%	\item $pqr \in \mathbb{N}(\Delta)$ (with $p < q < r$ primes) can be completely characterized by a pair of arithmetic conditions involving the polynomial $p^{2} - p + 1$ and the triples $(q, r) \in \{(p+2, 2p+1), (2p-1, 3p-2)\}$; in particular, $105$, $385$ and $1729$ are the only such products with $p \leq 7$.
%\end{itemize}

\subsection{Returning to graphs}

Obstructions to the $\Delta$ condition translate back into obstructions at the level of the factor graph. As a sample of this translation, we close \Cref{sec:Num.sin.peop.Delta} with the following structural consequence: if $n \in \mathbb{N}$ is neither a perfect square nor an element of $\mathbb{N}(\Delta)$, then every $n$-simple induced cycle of $\Phi(S)$ has length $4$ (\Cref{teo:return_to_Phi(S)}). 

%\subsection{Organization}
%
%The paper is organized as follows. \Cref{sec:preliminaries} collects the basic preliminaries on split graphs, factor graphs and the 2-switch-degree that will be used throughout. \Cref{sec:bridge} establishes the bridge between Graph Theory and Number Theory: it contains the two directions of \Cref{thmintro:A} together with several corollaries relating $D^{*}_{n}$ to $n$-simple split graphs.
%%and $\varepsilon$-linear split graphs. 
%\Cref{sec:primitives} is devoted to $\Delta$-primitive numbers, their basic properties and the semigroup structure of $\mathbb{N}(\Delta^{2})$. \Cref{sec:polynomials} develops the polynomial method and deduces the infinitude of $\Delta$-primitives. Finally, \Cref{sec:obstructions} presents the negative side of the theory, culminating in the characterization of the products of three primes in $\mathbb{N}(\Delta)$ and in the cycle-length corollary mentioned above. Open problems and conjectures are stated as they arise.

%%%%%%%%%%%%%%%%%%%%%%%%%%%%%%%%%%%%%%%%%%%%%%%%%%%%%%%%%

\section{$\Delta$ property and $n$-simple triangles in $\Phi(S)$}\label{sec:Graphs<->NumberTheory}

\subsection{Basic facts about $\Delta$-triples}

\begin{proposition}
	\label{cotasup.D_n^*}
	Let $n, x, y$ be natural numbers and such that $x\leq y\leq\sqrt{n}$. Then,
	\begin{equation}
		\label{eq15}
		\frac{n}{y}-y\leq\frac{n}{x}-x,
	\end{equation} 
	and equality holds if and only if $x=y$. In particular, we have
	\[ \max\big(D_n^*-\{n-1\}\big)\leq\frac{n}{2}-2. \]
\end{proposition}

\begin{proof}
	The hypothesis $x\leq y$ implies $-y\leq -x$ and $\frac{n}{y}\leq\frac{n}{x}$, which added together yield inequality \eqref{eq15}. To obtain the latter bound, notice that $\max(D_n^*)=n-1$, and then take $x=2$ in \eqref{eq15}.
\end{proof}

\begin{corollary}
	\label{cotasup.D^*capD^+}
	For all $n\in\mathbb{N}(\Delta)$:
	\[ \max(D_n^*\cap D_n^+)\leq \frac{2n}{3}-6. \]
\end{corollary}

\begin{proof}
	Since $D_n^*\cap D_n^+\neq\varnothing$, there exists a $\Delta$-triple $(x,y,z)$ for $n$ such that $\frac{n}{x}-x = \max(D_n^*\cap D_n^+)=m$.	Since $y,z\geq 3$, it follows from  \Cref{cotasup.D_n^*} that $m\leq 2\big(\frac{n}{3}-3\big)$.
\end{proof}

\begin{proposition}
	\label{Delta.prop.caracterizacion}
	A natural number $n$ satisfies the $\Delta$ condition if and only if there exists $(x,y,z)\in D_n^3$ such that $1<x<y\leq z<\sqrt{n}$ and
	\begin{equation}
		\frac{n}{x}-x = \frac{n}{y}-y + \frac{n}{z}-z.
		\label{ecuacion.Delta.prop}
	\end{equation}
\end{proposition}

A triple $(x,y,z)$ of divisors of $n$ satisfying all the requirements of  \Cref{Delta.prop.caracterizacion} is called a $\Delta$-\emph{triple} for $n$. Through elementary algebraic manipulations, \eqref{ecuacion.Delta.prop} can be rewritten as
\begin{equation}
	(xy+xz-yz)n=xyz(z+y-x).
	\label{ecuacion2.Delta.prop}	
\end{equation}

\begin{proof}
	A number $c$ belongs to $D^*_{n}\cap D^+_{n}$ if and only if $c\in D^*_{n}$ and $c=a+b$, for some $a,b\in D^*_{n}-\{0\}$. Since $a,b,c\in D^*_{n}$, we can write them as $a=\frac{n}{z}-z$, $b=\frac{n}{y}-y$ and $c=\frac{n}{x}-x$, where $x,y,z$ are divisors of $n$. As $a,b>0$, then also $c>0$. Thus, $x^2,y^2,z^2<n$, so $x,y,z<\sqrt{n}$. Since $c=a+b$ and $a,b>0$, clearly $a,b<c$, that is, $x<y,z$. By symmetry on the right-hand side of \eqref{ecuacion.Delta.prop}, we may assume without loss of generality that $y\leq z$. Finally, suppose $x=1$. Since $y,z\geq 2$, we have by \Cref{cotasup.D_n^*} that $n-1=c=a+b\leq 2\big(\frac{n}{2}-2\big)=n-4$, which is a contradiction. Therefore, it must be $x>1$.
\end{proof}

%$(2,3,3)$ and $(4,5,5)$ are $\Delta$-triples for 24 and 40, respectively: 
%\[ \frac{24}{2}-2=10=2\Big(\frac{24}{3}-3\Big), \  \frac{40}{4}-4=6=2\Big(\frac{40}{5}-5\Big). \]

%A curious consequence of  \Cref{cotasup.D_n^*} is the exotic characterization of perfect squares we now present.
%
%\begin{corollary}
%	Let $D^*_n +D^*_n =\{x+y:x,y\in D^*_n\}$. An integer $n>1$ is a perfect square if and only if $D^*_n \subsetneq D^*_n+D^*_n.$ 
%\end{corollary}
%
%\begin{proof}
%	If $n$ is a square, then $0\in D^*_n$ and thus $x+0\in D^*_n+D^*_n$ for all $x\in D^*_n$. Hence, $D^*_n \subset D^*_n+D^*_n$. To see that this inclusion is always strict, note that $n-1\in D^*_n$ but $2n-2=(n-1)+(n-1)\in (D^*_n+D^*_n)-D^*_n$, since $n>1$.
%	
%	For the converse, suppose $D^*_n \subset D^*_n+D^*_n$. This means, in particular, that $n-1\in D^*_n\cap(D^*_n+D^*_n)$. Then, there exist $y,z\in D_n$ such that $y^2,z^2\leq n$ and
%	\begin{equation}
%		\frac{n}{z}-z+\frac{n}{y}-y=n-1.
%		\label{ecuac.Delta.prop.x=1}	
%	\end{equation} 
%	If $y,z\geq 2$, then by  \Cref{cotasup.D_n^*}, we would have
%	\[ \frac{n}{z}-z, \ \frac{n}{y}-y \leq\frac{n}{2}-2, \]
%	and thus (\Cref{ecuac.Delta.prop.x=1}) could not hold. Therefore, either $y=1$ or $z=1$. If $y=1$, then $n=z^2$. If $z=1$, then $n=y^2$.
%\end{proof}

\begin{proposition}
	\label{prop:y/x.cota.universal}
	If $(x, y, z)$ is a $\Delta$-triple for $n$, then $1 < \frac{y}{x} < 2$. In particular, $x\notin D_y$
	.
\end{proposition}

\begin{proof}
	Since $(x,y,z)$ is a $\Delta$-triple for $n$, we have $1<y/x$, $y/z\le 1$ and \eqref{ecuacion2.Delta.prop}. Thus, we must have $xy+xz>yz$, because $z+y-x, xyz$ and $n$ are all positive. Hence, $\frac{y}{x}<1+\frac{y}{z}\le 2$. If $x\mid y$, then $y/x\in\mathbb{N}$. However, $(1,2)\cap\mathbb{N}=\varnothing$. 
\end{proof}

\begin{proposition}\label{prop:bound-z}  % <-- ADJUST LABEL
	If $(x,y,z)$ is a $\Delta$-triple for $n$, then $z<x(x+1)$.
	In particular, writing $z=ax+b$ with $a,b\in\mathbb{Z}$ and $0\le b<x$,
	we have $1\le a\le x$.
\end{proposition}

\begin{proof}
	By \eqref{ecuacion2.Delta.prop}, we have
	$xy+xz>yz$. Hence,
	\[ z<\frac{xy}{y-x}=x+\frac{x^2}{y-x}\le x+x^2. \]
	It is clear that $a\ge 1$. On the other hand,  $a\ge x+1$ implies $z=ax+b\ge (x+1)x+b>z+b$, which is absurd. Therefore, $a\le x$.
\end{proof}

\begin{corollary}\label{cor:finite-fixed-component}  
	For each $t\in\mathbb{N}$ with $t\ge 2$, the set
	\[
	\mathbb{N}(\Delta,t)=\{\,n\in\mathbb{N}(\Delta):\ \text{$t$ is a component
		of some $\Delta$-triple of $n$}\,\}
	\]
	is finite.
\end{corollary}

\begin{proof}
	Every $\Delta$-triple $(x,y,z)$ containing $t$ as a component satisfies $x\le t$. By \Cref{prop:y/x.cota.universal,prop:bound-z}, once $x$ is fixed there are at most $x-1$ admissible values for $y$ (in $\{x+1,\dots,2x-1\}$) and at most $x^2-1$ admissible values for $z$ (in $\{y,\dots,x^2+x-1\}$). Summing over $x \in \{2,\ldots,t\}$ gives a finite number of admissible triples, and by \eqref{ecuacion2.Delta.prop} each triple determines $n$ uniquely.
\end{proof}

\begin{proposition}
	Let $(x,y,z)$ be a $\Delta$-triple for $n$. If $z=cy+d$, with $c,d\in\mathbb{Z}$ and $0\le d<y$, then $1\le c<x$. 
\end{proposition}

\begin{proof}
	By \eqref{ecuacion2.Delta.prop}, we know that $z<xy/(y-x)$, and so $z<xy$. It is clear that $c\ge 1$. On the other hand, $c\ge x$ implies $z=cy+d\ge xy+d>z+d$, which is absurd. Therefore, $c<x$.
\end{proof}

\begin{proposition}\label{prop:bound-z-lcm}  % <-- ADJUST LABEL
	If $(x,y,z)$ is a $\Delta$-triple for $n$, then $z<\operatorname{lcm}(x,y)$.
	Equivalently, $z\gcd(x,y)<xy$.
\end{proposition}

\begin{proof}
	Let $g=\gcd(x,y)$. Since $g\mid x$ and $g\mid y$, we have $g\mid (y-x)$,
	so $y-x\ge g$. Using \eqref{ecuacion2.Delta.prop}: $z<xy/(y-x)\le xy/g=\operatorname{lcm}(x,y)$.
\end{proof}

%%%%%%%%%%%%%%%%%%%%%%%%%%%%%%%%%%%%%%%%%%%%%%%%%%%%%%

\subsection{Optimal bounds on $n$ in terms of $x$}\label{subsec:optimal-bounds} 

The bounds on $y$ and $z$ established in \Cref{prop:y/x.cota.universal,prop:bound-z} constrain the components of any $\Delta$-triple, but do not by themselves bound $n$. In this subsection we determine the optimal upper bound on $n$ in terms of $x$, separately for each of the two regimes $y=z$ and $y<z$. Both bounds rest on a single change of variables: setting $u = y-x$ and $v = z-x$, the equality \eqref{ecuacion2.Delta.prop} characterizing a $\Delta$-triple becomes
\begin{equation}
	\label{eq_delta_u=y-x__v=z-x}
	n \;=\; n(x,u,v) \;=\; \frac{x(x+u)(x+v)(x+u+v)}{x^{2}-uv},
\end{equation}
and the admissibility condition $xy+xz-yz \ge 1$ reduces to $uv \le x^{2}-1$. By \Cref{prop:y/x.cota.universal,prop:bound-z}, the integers $u,v$ satisfy $1 \le u \le v \le x^{2}-1$. The two regimes $y=z$ and $y<z$ correspond to $u=v$ and $u<v$, respectively. In each, the maximization of $n$ reduces to elementary monotonicity arguments.

\begin{theorem}\label{prop:max-duplicated}
	Let $x$ be an integer $\geq 2$ and let $F(x)=x(2x-1)(3x-2)$. 
	\begin{enumerate}
		\item If $(x,y,y)$ is a $\Delta$-triple for $n$, then $n\le F(x)$.
		
		\item $n=F(x)$ if and only if $y=2x-1$.
		
		\item $(x,2x-1,2x-1)$ is a $\Delta$-triple for $F(x)$.
	\end{enumerate}
\end{theorem}

\begin{proof}
	\begin{enumerate}[(1).]
		\item Setting $v = u$ in \eqref{eq_delta_u=y-x__v=z-x} gives
		\[
		n(u) \;=\; \frac{x(x+u)(x+2u)}{x-u},
		\]
		where $u\in[x-1]$. We show $n$ is strictly increasing in $u$ on this range. For $1 \le u \le x-2$, we have
		\[
		\frac{n(u+1)}{n(u)} \;=\; \frac{x+u+1}{x+u} \cdot \frac{x+2u+2}{x+2u} \cdot \frac{x-u}{x-u-1} \;>\; 1,
		\]
		since each of the three factors exceeds $1$. Hence, the maximum is attained at $u = x-1$, that is, $y = 2x-1$, where $n = F(x)$.
		
		\item Proved in (1).
		
		\item Let $y=2x-1$ and $n=F(x)$. Then, $x,y\in D_n$, $1=2x-y$, and $3x-2=2y-x$. Substituting all of this into the identity $1\cdot n=F(x)$ gives $(2x-y)n=xy(2y-x)$, which is exactly \eqref{ecuacion2.Delta.prop} for $z=y$. Since $(2x-1)^2<F(x)$ holds for all $x\geq 2$, we have  $2\leq x <y\leq z<\sqrt{n}$. Hence, $(x,y,y)$ is a $\Delta$-triple for $n$ by  \Cref{Delta.prop.caracterizacion}.
	\end{enumerate}
\end{proof}

\begin{theorem}\label{prop:max-generic}
	Let $x$ be an integer $\geq 2$ and let $F(x)=x^{2}(x+1)^{2}(x^{2}+x-1)$. 
	\begin{enumerate}
		\item If $(x,y,z)$ is a $\Delta$-triple for $n$ with $y<z$, then $n\le F(x)$.
		
		\item $n=F(x)$ if and only if $(y,z)=(x+1,\,x^{2}+x-1)$.
		 
		\item $(x,x+1,x^2+x-1)$ is a $\Delta$-triple for $F(x)$.
	\end{enumerate}
\end{theorem}

\begin{proof}
	\begin{enumerate}[(1).]
		\item Recall that $u=y-x$ and $v=z-x$, with $1 \le u < v$ and $2\le uv \le x^{2}-1$ (see \eqref{eq_delta_u=y-x__v=z-x}). Let $p = uv$. Since $(u-1)(v-1) \ge 0$, we have $u+v\le p+1$. Thus,
		\begin{align*}
			(x+u)(x+v) &\;\le\; x^{2} + x(p+1) + p \;=\; (x+1)(x+p),\\
			x+u+v &\;\le\; x+p+1.
		\end{align*}
		Using these inequalities to bound \eqref{eq_delta_u=y-x__v=z-x} yields
		\[
		n \;\le\; \frac{x(x+1)(x+p)(x+p+1)}{x^{2}-p} \;=\; f(p).
		\]
		For $2 \le p \le x^{2}-2$ we have
		\[
		\frac{f(p+1)}{f(p)} \;=\; \frac{x+p+2}{x+p}\cdot\frac{x^{2}-p}{x^{2}-p-1} \;>\; 1,
		\]
		since both factors exceed $1$. Hence, $f$ is strictly increasing on $\{2,\dots,x^{2}-1\}$, and $n\le f(x^2-1)=F(x)$.
		
		\item We will show that $(y,z) \ne (x+1,\,x^{2}+x-1)$ implies $n<F(x)$ (the converse is straightforward to check). By \Cref{prop:y/x.cota.universal,prop:bound-z}, we have $y \ge x+1$ and $z \le x^{2}+x-1$. 
		
		If $y \ge x+2$, then $z \ge y+1 \ge x+3$, so $(y-x-1)(z-x-1)\ge 1\cdot 2>0$,	which expands to $u+v < p + 1$. Hence, $(x+u)(x+v)<(x+1)(x+p)$ and $x+u+v<x+p+1$. Therefore, $n < f(p) \le f(x^{2}-1)=F(x)$. 
		
		If $y = x+1$ and $z \le x^{2}+x-2$, then $p = z-x \le x^{2}-2$, so by the strict monotonicity of $f$: $n\le f(p)<f(x^2-1)=F(x)$. 
		
		\item Substituting $y=x+1$ and $z=x^2+x-1$ into
		\eqref{ecuacion2.Delta.prop} and simplifying gives $\bigl(x(x+1)-z\bigr)\,n=x(x+1)\,z\,(z+1)$.
		Since $x(x+1)-z=1$, it follows that $n=x(x+1)\,z(z+1)=F(x)$.
		We now verify the hypotheses of \Cref{Delta.prop.caracterizacion}: $x$, $x+1$, $x^2+x-1\in D_n$ by construction; $1<x<x+1\le x^2+x-1$ holds for $x\ge 2$; and $z^2<n$ because
		$n-z^2=z(x^2y^2-z)>0$ (recall that $z<xy$).
	\end{enumerate}
\end{proof}

Observe from \Cref{prop:max-generic}(3) that the family of $\Delta$-triples $\{(x, x+1, x^2 + x - 1):x\ge 2\}$ saturates the bound $z<x(x+1)$ of \Cref{prop:bound-z}.

\subsection{Arithmetic constraint when $y=z$}\label{subsec:descent}

The optimal bounds of \Cref{subsec:optimal-bounds} measure how large $n$ can be; they say nothing about how its prime structure is organized. We close the section by recording the arithmetic counterpart for the duplicated regime $y = z$. Writing $x=ad$ and $y=bd$, where $d = \gcd(x, y)$ and $\gcd(a,b)=1$, shows that \eqref{ecuacion2.Delta.prop} becomes a descent identity that ties $n$ to the square factor $d^2$. Moreover, we obtain a sharp divisibility constraint: $\gcd(2a-b,ab(2b-a))|6$.

\begin{theorem}
	\label{thm:duplicated-descent}
	Let $(x, y, y)$ be a $\Delta$-triple for $n$, set 
	$d = \gcd(x, y)$, and write $x = da$, $y = db$, with $\gcd(a, b) = 1$. Then:
	\begin{equation}
		\label{eq:duplicated-descent}
		(2a - b)\, n \,=\, d^2\, ab\,(2b - a),
	\end{equation}
	\[
	\gcd\bigl(2a - b,\ ab(2b - a)\bigr) \,\Big|\, 6,
	\]
	and
	\[
	\frac{2a - b}{\gcd(2a - b,\ 6)} \;\Big|\; d^2.
	\]
\end{theorem}

\begin{proof}
	Substituting $z = y$ in \eqref{ecuacion2.Delta.prop} yields $(2x - y)\, n = xy\,(2y - x)$. Replacing $x = da$ and $y = db$ gives \eqref{eq:duplicated-descent}. For the $\gcd$ bound, we estimate $\gcd(2a-b,\, ab(2b-a))$ by using that:
	\[
	\gcd\bigl(2a - b, ab(2b - a)\bigr) 
	\,\Big|\, 
	\gcd(2a - b, a) \gcd(2a - b, b) \gcd(2a - b, 2b - a).
	\]
	Any common divisor $g$ of $2a - b$ and $a$ divides $2a - (2a - b) = b$. Hence $g \mid \gcd(a, b) = 1$. So, $\gcd(2a - b, a) = 1$.
	
	Any common divisor $g$ of $2a - b$ and $b$ divides $(2a - b) + b = 2a$. Hence $g \mid \gcd(b, 2a)$. Since $\gcd(a, b) = 1$, we have $\gcd(b, 2a) = \gcd(b, 2)$, which divides $2$. Thus, $\gcd(2a - b, b) \mid 2$.
	
	Any common divisor $g$ of $2a - b$ and $2b - a$ divides both $2(2a - b) + (2b - a) = 3a$ and $(2a - b) + 2(2b - a) = 3b$. Hence $g \mid \gcd(3a, 3b) = 3$. Therefore, $\gcd(2a - b, 2b - a) \mid 3$.
	
	Combining these three facts: $\gcd(2a - b, ab(2b - a)) \mid 1 \cdot 2 \cdot 3 = 6$.
	
	Finally, from \eqref{eq:duplicated-descent} we have $(2a - b) \mid d^2\, ab\,(2b - a)$. Writing $u_1 = \gcd(2a - b, ab(2b - a))$ and $u_2 = (2a - b)/u_1$, we have $\gcd\bigl(u_2, ab(2b - a)/u_1\bigr) = 1$, so $u_2 \mid d^2$. Since $u_1 \mid \gcd(2a-b,\, 6)$ (because $u_1 \mid 2a-b$ and $u_1 \mid 6$), it follows that $(2a-b)/\gcd(2a-b,\,6) \mid u_2 \mid d^2$.
\end{proof}

%%%%%%%%%%%%%%%%%%%%%%%%%%%%%%%%%%%%%%%%%%%%%%%%%%%%%%%%

\subsection{From Graphs to Numbers}

%We now show the connection that exists between split graphs and the $\Delta$ property. 
As announced in \Cref{sec:intro}, the following theorem, is very important because it provides a bridge from Graph Theory to Number Theory. Its proof is surprisingly simple.

\begin{theorem}
	\label{triang.n-simple&tipo0.implica.Delta.prop}
	Let $S$ be a split graph.
	\begin{enumerate}
		\item If $\sigma_{uv}\neq 0$, then $|d_u-d_v|\in D_{\sigma_{uv}}^*$.
		\item If $\Phi(S)$ is an $n$-simple triangle of type 0 (see \Cref{triangulos.permitidos}), then $n\in\mathbb{N}(\Delta)$.
	\end{enumerate}
\end{theorem}
\begin{figure}[h]
	\centering
	\begin{tikzpicture}[scale=2, every node/.style={circle, draw, thick, minimum size=6pt, inner sep=2pt}]
		
		% --- Primer triángulo (K_3 no dirigido) ---
		\node [label=above left:{$\Phi(S)$}](x1) at (0,0) {$a$};
		\node (x2) at (1,0) {$c$};
		\node (x3) at (0.5, {0.5*sqrt(3)}) {$b$};
		
		\draw[thick] (x1) -- (x2) node[midway, draw=none, fill=white, inner sep=1pt] {$n$};
		\draw[thick] (x2) -- (x3) node[midway, draw=none, fill=white, inner sep=1pt] {$n$};
		\draw[thick] (x3) -- (x1) node[midway, draw=none, fill=white, inner sep=1pt] {$n$};
		
		\draw[bend left=15] (x1) to (x2) ;
		\draw[bend right=15] (x1) to (x3) ;
		\draw[bend left=15] (x2) to (x3) ;
		
		% --- Segundo triángulo (dirigido) ---
		\node (a) at (2,0) {$d_a$};
		\node [label=above right:{$S$}](c) at (3,0) {$d_c$};
		\node (b) at (2.5, {0.5*sqrt(3)}) {$d_b$};
		
		\draw[->, thick, >=stealth, line width=2pt, scale=1.5] (a) -- (c); % base ac
		\draw[->, thick, >=stealth, line width=2pt, scale=1.5] (a) -- (b);
		\draw[->, thick, >=stealth, line width=2pt, scale=1.5] (b) -- (c);
	\end{tikzpicture}
	\caption{Hypothesis of  \Cref{triang.n-simple&tipo0.implica.Delta.prop}.}
\end{figure}

\begin{proof}
	(1). Since $0\neq\sigma_{uv}=(d_u-\eta_{uv})(d_v-\eta_{uv})$, we have that $d_u-\eta_{uv}$ and $d_v-\eta_{uv}$ are complementary divisors of $\sigma_{uv}$. Therefore,
	\[ |(d_u-\eta_{uv})-(d_v-\eta_{uv})|=|d_u-d_v|\in D_{\sigma_{uv}}^*. \]
	
	(2). From (1), we have $|d_u-d_v|\in D_n^*$ for every edge $uv\in\Phi$. Since $\vec{\Phi}=\Delta_0$, it follows that $d_c>d_a$, $d_b>d_a$, and $d_c>d_b$. Moreover, the arcs of $\vec{\Phi}$ indicate that the degree increase (in $S$) from $a$ to $c$ equals the increase from $a$ to $b$ plus the increase from $b$ to $c$. Hence,
	\[ d_c-d_a =(d_b-d_a)+(d_c-d_b)\in D^*_n \cap D^+_n, \] 
	which shows that $D^*_n \cap D^+_n\neq\varnothing$.
\end{proof}

% \Cref{triang.n-simple&tipo0.implica.Delta.prop} has some corollaries that relate the set $D_n^*$ to split graphs, especially those whose associated factor graphs are $n$-simple.
%and $\varepsilon$-linear.

\begin{lemma}
	\label{S.squarefree.implica.triang_tipo0}
	Let $S$ be a split graph and let $T$ be a triangle in $\Phi(S)$. If $\sigma_{uv}$ is not a positive square for every $uv\in E(T)$, then $\vec{T}=\Delta_0$.
\end{lemma}

\begin{proof}
	If $T$ is a triangle and $\vec{T}\neq\Delta_0$, then $\vec{T}$ must have arcs of the form $xy$ and $yx$, for some $x,y\in V(T)$ (recall \Cref{triangulos.permitidos}). Then, $\sigma_{xy}>0$ and $d_x=d_y$. Consequently, $\sigma_{xy}=(d_x-\eta_{xy})^2$. 
\end{proof}

\begin{corollary}
	\label{triang.n-simple&n-nonsquare.implica.Delta.prop}
	Let $\Phi(S)$ be an $n$-simple triangle. If $n$ is not a square, then $n\in\mathbb{N}(\Delta)$. 
\end{corollary}

\begin{proof}
	 By \Cref{S.squarefree.implica.triang_tipo0}, $\Phi$ must be of type 0. Then, by \Cref{triang.n-simple&tipo0.implica.Delta.prop}(2), we conclude that $n\in\mathbb{N}(\Delta)$.
\end{proof}

\begin{corollary}
	\label{|d_u-d_v|<=sigma_uv-1}
	If $\sigma_{uv}\neq 0$, then $|d_u-d_v|\leq\sigma_{uv}-1$.
\end{corollary}

\begin{proof}
	Since $\max(D_n^*)= n-1$ for every $n\in\mathbb{N}$, the claim follows immediately from \Cref{triang.n-simple&tipo0.implica.Delta.prop}(1).
\end{proof}

\begin{corollary}
	If $\Phi(S)$ is $n$-simple, then 
	\begin{equation}
		\label{eq16}
		\{|d_u-d_v|: \sigma_{uv}\neq 0\}\subset D_n^*.
	\end{equation}
	In particular, if $\Phi(S)$ has two adjacent vertices with the same degree in $S$, then $n$ is a perfect square. 
\end{corollary}

\begin{proof}
	The inclusion \eqref{eq16} is a direct consequence of \Cref{triang.n-simple&tipo0.implica.Delta.prop}(1). If $\Phi$ has two adjacent vertices with the same degree in $S$, then by \eqref{eq16}, we have $0\in D_n^*$, which is equivalent to $n$ being a square. 
\end{proof}

\subsection{From Numbers to Graphs}

A split graph $(S,K,I)$ is said to be \textit{balanced} if $|K|=\omega(S)$ and $|I|=\alpha(S)$, where $\omega(S)$ is the \textit{clique number} of $S$ and $\alpha(S)$ is the \textit{independence number} of $S$ (see \cite{cheng2016split}). Otherwise, we say that $S$ is \textit{unbalanced}. $S$ is unbalanced if and only if $(|K|,|I|)\in\{(\omega(S),\alpha(S)-1), (\omega(S)-1,\alpha(S))\}$ (\cite{cheng2016split}, Theorem 7). To prove the following lemma, we use some notations, concepts and results about unbalanced split graphs from Section 4 of \cite{jaume2025nullspace} (in particular, about \textit{swing} vertices).

\begin{lemma}
	\label{phi.sin.aislados.implica.balanced}
	Let $(S,K,I)$ be a split graph such that $K=\bigcup_{v\in I}N_v$. If $\Phi(S)$ has no isolated vertices, then $S$ is balanced.
\end{lemma}

\begin{proof}
	Suppose, for contradiction, that $S$ is unbalanced. Then, $|I|\in\{\alpha(S),$ $\alpha(S)-1\}$. Let $W(S)$ be the set of swing vertices of $S$.
	
	If $|I|=\alpha(S)$, then $|K|=\omega(S)-1$. Since $I$ is a maximum independent set of an unbalanced split graph, $I\cap W(S)\neq\varnothing$ by Corollary 4.8 of \cite{jaume2025nullspace}. Pick $w\in W(S)\cap I$. By Theorem 4.12(1) of \cite{jaume2025nullspace}, $d_w=\omega(S)-1=|K|$, and since $N_w\subset K$, this forces $N_w=K$. For every $v\in I-w$, $N_v\subset K=N_w$ gives $\eta_{vw}=d_v$, hence $\sigma_{vw}=0$. Thus, $w$ is isolated in $\Phi(S)$, contradicting the hypothesis.
	
	If $|I|=\alpha(S)-1$, then $|K|=\omega(S)$, and any maximum independent set has the form $I\cup x$ for some $x\in K$. Such $x$ is not adjacent to any vertex of $I$, so $x\notin\bigcup_{v\in I}N_v=K$, a contradiction.
\end{proof}

A vertex $v$ of a graph $G$ is said to be \textit{active} if $v$ is involved in some 2-switch on $G$ (see \cite{vnsigma.2switch.degree}). Otherwise, $v$ is \textit{inactive}. We say that $G$ is active if all of its vertices are active. We say that $G$ is \textit{decomposable} (with respect to the Tyshkevich composition $\circ$, see \cite{tyshkevich2000decomposition}) if $G=S\circ H$ for some split graph $S$ and graph $H$, both with at least one vertex. Otherwise, $G$ is \textit{indecomposable}. Every nontrivial indecomposable graph is active. $G=S\circ H$ is active if and only if $S$ and $H$ are active. An active split graph is indecomposable if and only if $\Phi(S)$ is connected (see \cite{vnsigma.factor.graph}).

\begin{lemma}[\cite{vnsigma.factor.graph}]
	\label{Phi.completo.implica.inactivos.universales}
	Let $(S,K,I)$ be a balanced split graph such that $|I|\geq 2$ and $\Phi(S)$ is complete. If $x$ is an inactive vertex in $S$, then $x$ is universal (i.e., $N_x=V(S)-x$).
\end{lemma}

%\begin{proposition}
%	\label{split.caract.vert.activos}
%	If $ (S, K, I) $ is a split graph, then the following holds:
%	\begin{enumerate}
%		\item $ U = \bigcap_{v \in I} N_S(v) $ is the set of universal vertices of $ S $;
%		\item $ u \in I $ is active in $ S $ if and only if there exists an $ x \in I $ such that $ N_S(x) - N_S(u) \neq \varnothing $ and $ N_S(u) - N_S(x) \neq \varnothing $;
%		\item $ u \in I $ is inactive in $ S $ if and only if, for every $ v \in I $, either $ N_S(v) \subset N_S(u) $ or $ N_S(u) \subset N_S(v) $;
%		\item $ u \in K $ is inactive in $ S $ if and only if, for every $ v \in K $, either $ N_S(v) \cap I \subset N_S(u) \cap I $ or $ N_S(u) \cap I \subset N_S(v) \cap I $; 
%		\item A vertex of $ S $ is inactive in $ S $ if and only if its neighborhood is comparable by inclusion with all other neighborhoods in its partition;
%		\item If $ w $ is a swing vertex in the split graph $ (S, K, I) $, then $ w \notin act(S) $.
%		\item If $ S $ is active, then $ S $ is balanced.
%	\end{enumerate}
%\end{proposition}

If \Cref{triang.n-simple&tipo0.implica.Delta.prop} can be thought of as a bridge from Graphs to Numbers, the following result goes in the opposite direction, that is, from Number Theory back to Graph Theory. For this reason, we consider it another key theorem of this section.

\begin{theorem}
	\label{n.con.propDelta.implica.existencia.de.T_0.n-simple}
	Let $(x,y,z)$ be a $\Delta$-triple for $n$. For each integer $k\ge z$, there exists a balanced split graph $(S,K,I)$ with $I=\{a,b,c\}$, such that $d_a=k$ and $\Phi(S)$ is an $n$-simple triangle of type 0 (see \Cref{triangulos.permitidos}). The graph $S$ satisfies:
	\begin{enumerate}
		\item $d_b=d_a+\frac{n}{z}-z$ and $d_c=d_a+\frac{n}{x}-x$;
		
		\item $\eta_{ab}=d_a-z$, $\eta_{bc}=d_a+\frac{n}{z}-z-y>0$, and $\eta_{ac}=d_a-x>0$;
		
		\item $\eta_{abc}=|N_a\cap N_b\cap N_c|\geq\max\{0,\,d_a-x-z\}$;
		
		\item $|K|=\frac{n}{x}+z+y+\eta_{abc}$;
		
		\item if $d_a=z$, then $S$ is active and indecomposable;
		
		\item if $S$ is active, then $S$ is indecomposable and $d_a\leq x+z$.
	\end{enumerate}
\end{theorem}

\begin{proof}
	We construct a split graph $(S,K,I)$ with $K=N_a\cup N_b\cup N_c$. Since $K=\bigcup_{v\in I}N_v$ and since $\Phi(S)$ has no isolated vertices, we have that $S$ is balanced by  \Cref{phi.sin.aislados.implica.balanced}. For $S$ to have the required properties it suffices to express $d_a,d_b,d_c,\eta_{ab},\eta_{bc}$ and $\eta_{ac}$ in terms of $n,x,y$ and $z$. For each $\{u,v\}\subset I$, we have $n=\sigma_{uv}=(d_u-\eta_{uv})(d_v-\eta_{uv})$, so $d_u-\eta_{uv}$ and $d_v-\eta_{uv}$ are complementary divisors of $n$. We set
	\begin{align*}
		d_a-\eta_{ab}&=z, & d_b-\eta_{ab}&=n/z,\\
		d_b-\eta_{bc}&=y, & d_c-\eta_{bc}&=n/y,\\
		d_a-\eta_{ac}&=x, & d_c-\eta_{ac}&=n/x.
	\end{align*}
	These relations make $S$ $n$-simple by construction and, treating $d_a$ as a free parameter, we solve for the remaining quantities, obtaining (1) and (2). In particular, $d_a<d_b<d_c$, so $\vec{\Phi}(S)=\Delta_0$ provided $d_a>0$, which is guaranteed by $d_a-z=\eta_{ab}\ge 0$.
	
	\begin{enumerate}[(1).]
		\item Immediate from the relations above.
		
		\item The identities immediately follows from the relations above. Since $d_a\geq z>x$, we have $\eta_{ac}>0$. On the other hand, notice that $y,z<\sqrt{n}$ implies $yz<n$, i.e., $\frac{n}{z}-y>0$. Therefore, $d_a\ge z>z-(\frac{n}{z}-y)$, which means $\eta_{bc}>0$.
		
		\item Combining (2) with the Inclusion-Exclusion Principle for two sets and some basic set identities, we deduce:
			\[
			d_a\ge |N_a\cap(N_b\cup N_c)|=|(N_a\cap N_b)\cup(N_a\cap N_c)|= 
			\]	
			\[
			\eta_{ab}+\eta_{ac}-\eta_{abc}=2d_a-x-z-\eta_{abc}.
			\]
			Consequently, $\eta_{abc}\geq d_a-x-z.$

		\item It follows by combining (1) and (2) with the Inclusion-Exclusion Principle for three sets. 
		
		\item If $d_a=z$, then $\eta_{ab}=0$ by (2), and hence $\eta_{abc}=0$. Therefore, $S$ is active by  \Cref{Phi.completo.implica.inactivos.universales}. Since $\Phi(S)$ is connected, $S$ is indecomposable by Theorem 2.7 in \cite{vnsigma.factor.graph}.
		
		\item If $S$ is active, it has no universal vertices (which are always inactive), so $\eta_{abc}=0$ (see \cite{vnsigma.2switch.degree})
		%by  \Cref{split.caract.vert.activos}. 
		If $d_a>x+z$, then, by (3), $\eta_{abc}\geq 1$, i.e. there is a universal vertex in $K$, contradicting that $S$ is active. Hence $d_a\leq x+z$. Since $\Phi(S)$ is connected, $S$ is indecomposable by Theorem 2.7 in \cite{vnsigma.factor.graph}. \qedhere
	\end{enumerate}
\end{proof}

Let us apply \Cref{n.con.propDelta.implica.existencia.de.T_0.n-simple}
to the $\Delta$-triple $(2,3,3)$ for $n=24$, taking $d_a=3$. We
obtain an indecomposable split graph $S$ with $|K|=18$, $d_b=8$, $d_c=13$, $\eta_{ab}=0$, $\eta_{bc}=5$ and $\eta_{ac}=1$. The following neighborhoods are compatible with these parameters:
if $K=[18]$, define $N_a=[3], N_b=\{4,\ldots,11\}$, and $N_c=\{3,\ldots,8,12,\ldots,18\}$. Thus, $N_a\cap N_b=\varnothing, N_b\cap N_c=\{4,\ldots,8\}$, and $N_a\cap N_c=\{3\}$. It can be easily checked that $\Phi(S)$ is a $24$-simple triangle by direct computation of the   $\sigma_{uv}$'s.
%Indeed, $\sigma_{ab}=(3-0)(8-0)=24$, $\sigma_{bc}=(8-5)(13-5)=24$, and $\sigma_{ac}=(3-1)(13-1)=24$. 
In  \Cref{split.24simple.Delta.tipo0}, we show a simplified version of the split graph $S$ we have just constructed, omitting all edges between clique vertices (in black).

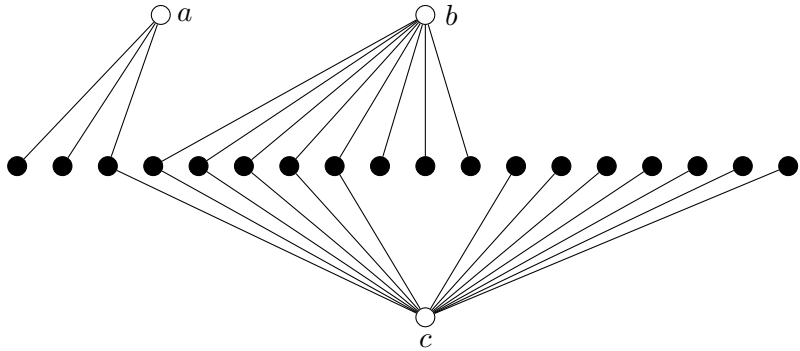
\begin{figure}[ht]
	\centering
	\begin{tikzpicture}[scale=0.5, every node/.style={circle, draw, minimum size=0.25cm,inner sep=1pt}]
		% Coordinates for independent nodes
		\coordinate (center) at (0,0);
		
		% Nodes from {1,...,18} on a horizontal axis
		\foreach \i [count=\j from -9] in {1,...,18} {
			\node[fill=black] (K\i) at (\j*1.2, 0) {};
		}
		
		% Independent nodes: a, b, c
		\node[label=right:{$a$}] (a) at (-7, 4) {};
		\node[label=right:{$b$}] (b) at (0, 4) {};
		\node[label=below:{$c$}] (c) at (0, -4) {};
		
		% Neighborhood of "a"
		\foreach \i in {1,2,3} {
			\draw (a) -- (K\i);
		}
		
		% Neighborhood of "b"
		\foreach \i in {4,5,6,7,8,9,10,11} {
			\draw (b) -- (K\i);
		}
		
		% Neighborhood of "c"
		\foreach \i in {3,4,5,6,7,8,12,13,14,15,16,17,18} {
			\draw (c) -- (K\i);
		}
	\end{tikzpicture}
	\caption{Application of  \Cref{n.con.propDelta.implica.existencia.de.T_0.n-simple} to the case $n=24$.}
	\label{split.24simple.Delta.tipo0}
\end{figure}

%As we have just seen, all parameters in  \Cref{n.con.propDelta.implica.existencia.de.T_0.n-simple} defining $S$ depend on $d_a$, the degree of vertex $a$ in $S$. Indeed, according to that theorem, it is enough to fix any value $d_a\geq z$ to obtain essentially different balanced split graphs for different choices of $d_a$. For $d_a=z$, \Cref{n.con.propDelta.implica.existencia.de.T_0.n-simple}(5) ensures that $S$ is active. However, \Cref{n.con.propDelta.implica.existencia.de.T_0.n-simple}(6) restricts the number of active graphs we can construct to a finite quantity, since $S$ will have universal vertices if $d_a>x+z$. These observations are summarized in the following corollary.

\begin{corollary}
	\label{cor_finite_number_active_split}
	For each $n\in\mathbb{N}(\Delta)$:
	\begin{enumerate}
		\item there exists an infinite number of balanced split graphs $S$ such that $\Phi(S)$ is an $n$-simple triangle of type 0;
		\item there exists a finite nonzero number of active and indecomposable split graphs $S$ such that $\Phi(S)$ is an $n$-simple triangle of type 0.
	\end{enumerate}
\end{corollary}

\begin{proof}
	It immediately follows from \Cref{n.con.propDelta.implica.existencia.de.T_0.n-simple}.
\end{proof}

%%%%%%%%%%%%%%%%%%%%%%%%%%%%%%%%%%%%%%%%%%%%%%%%%%%%%%%%%%%

\section{$\Delta$-primitive numbers}\label{sec:Delta.prim}

%In this section, we observe that numbers of the form $\alpha^2n$, with $n\in\mathbb{N}(\Delta)$ and $\alpha\in\mathbb{N}$, all have the $\Delta$ property. This fact motivates the search for those members of $\mathbb{N}(\Delta)$ that are not nontrivial square multiples of any other element in $\mathbb{N}(\Delta)$. This leads to the concept of a $\Delta$-primitive number. We see that every element of $\mathbb{N}(\Delta)$ is a square multiple of some primitive, and that the set $\mathbb{N}(\Delta^2)$, consisting of all squares with the $\Delta$ property, forms an abelian semigroup under ordinary multiplication. Subsequently, we exhibit infinite examples of $\Delta$ squares that admit more than one representation as $\alpha^2m$, where $m$ is $\Delta$-primitive. Finally, we construct a polynomial generator of numbers with the $\Delta$ property, through which we prove that there are infinitely many $\Delta$-primitives. \\

\begin{proposition}
	\label{multiplos.cuadr.prop.delta}
	If $n\in\mathbb{N}(\Delta)$, then $\alpha^2 n\in\mathbb{N}(\Delta)$ for all $\alpha\in\mathbb{N}$. In particular, every odd power of $n$ satisfies the $\Delta$ condition.
\end{proposition}

\begin{proof}
	It suffices to observe that \eqref{ecuacion.Delta.prop} can be rewritten as
	\begin{equation*}
		\frac{\alpha^2 n}{\alpha z}-\alpha z + \frac{\alpha^2 n}{\alpha y}-\alpha y =\frac{\alpha^2 n}{\alpha x}-\alpha x.
	\end{equation*}
	Since $(\alpha x,\alpha y,\alpha z)\in D_{\alpha^2n}^3$, and since $1<x<y\leq z<\sqrt{n}$ implies
	\[ 1<\alpha x<\alpha y\leq\alpha z<\alpha\sqrt{n}=\sqrt{\alpha^2n}, \]
	we conclude that $\alpha^2 n\in\mathbb{N}(\Delta)$. In particular, taking $\alpha=n^\beta$, we see that $(n^\beta)^2 n=n^{2\beta+1}\in\mathbb{N}(\Delta)$ for all $\beta\in\mathbb{N}$.
\end{proof}

Thanks to  \Cref{multiplos.cuadr.prop.delta}, we see that $|\mathbb{N}(\Delta)|=\infty$. Another immediate consequence of this proposition is that if $n\in\mathbb{N}(\Delta)$ and $n^{2\beta}\in\mathbb{N}(\Delta)$ for some $\beta\in\mathbb{N}$, then $\{n^k:k\in[2\beta,\infty)\cap\mathbb{N}\}\subset\mathbb{N}(\Delta)$. For example, both $84$ and $84^2$ have the $\Delta$ property, so all higher powers of 84 do as well.

\Cref{multiplos.cuadr.prop.delta} allows us to trivially obtain infinitely many elements of $\mathbb{N}(\Delta)$ from others. This motivates the search and study of those numbers with the $\Delta$ property that cannot be obtained in this way. We say that a natural number $n$ is $\Delta$-\textit{primitive} if $n\in\mathbb{N}(\Delta)$ and there is no pair of numbers $\alpha,m\geq 2$ such that $n=\alpha^2m$ and $m\in\mathbb{N}(\Delta)$. For example, 24 and 40 are $\Delta$-primitive numbers. On the other hand, 96 is not, since we can write it as $2^2 \cdot 24$.

%A natural number $n$ is said to be \textbf{square-free} if there is no prime number $p$ such that $p^2|n$. In other words, all primes appearing in the factorization of $n$ have exponent 1. 
Clearly, if $n\in\mathbb{N}(\Delta)$ and $n$ is square-free, then $n$ is $\Delta$-primitive. Examples of square-free numbers satisfying the $\Delta$ condition are 105 and 385. Therefore, 105 and 385 are $\Delta$-primitive. Moreover, 105 is the smallest odd number with the $\Delta$ property.

\begin{proposition}
	\label{n=alfa^2 m_m.delta.prim}
	Every $n\in\mathbb{N}(\Delta)$ is either $\Delta$-primitive or can be written as $\alpha^2m$ for some $\alpha\geq 2$ and some $\Delta$-primitive $m$.
\end{proposition}

\begin{proof}
	If $n\in\mathbb{N}(\Delta)$ but is not primitive, then $n=\alpha_1^2m_1$ for some $\alpha_1\geq 2$ and some $m_1\in\mathbb{N}(\Delta)$. If $m_1$ is primitive, we are done. Otherwise, $m_1=\alpha_2^2m_2$ for some $\alpha_2\geq 2$ and $m_2\in\mathbb{N}(\Delta)$. If $m_2$ is primitive, we are done. Otherwise, $m_2=\alpha_3^2m_3$, and so on. Continuing this way generates a strictly decreasing sequence $\{m_1,m_2,m_3,\ldots\}\subset\mathbb{N}(\Delta)$. This process cannot continue indefinitely, so there must exist some $k\in\mathbb{N}$ such that $m_k$ is $\Delta$-primitive, hence $n=(\alpha_1\cdots\alpha_k)^2m_k$.  
\end{proof}

The notion of $\Delta$-primitivity, together with the decomposition of \Cref{n=alfa^2 m_m.delta.prim}, raises three natural questions:
\begin{enumerate}
	\item Is the decomposition of \Cref{n=alfa^2 m_m.delta.prim} unique?
	
	\item Are there infinitely many $\Delta$-primitive numbers?
	
	\item Are there infinitely many $\Delta$-primitive squares?
\end{enumerate}
In the remainder of this section we address (1) (\Cref{prop:coprime-uniqueness}) and (2) (\Cref{Delta-primitivos.son.infinitos}), while (3) remains open (Conjecture \ref{conj:inf_many_delta-squares}).

%%%%%%%%%%%%%%%%%%%%%%%%%%%%%%%%%%%%%%%%%%%%%%%%%%%%%%%%

\subsection{Squares with the $\Delta$ property}

An interesting consequence of  \Cref{multiplos.cuadr.prop.delta} is that the set $\mathbb{N}(\Delta^2)$ of all squares satisfying the $\Delta$ condition is closed under multiplication.

\begin{corollary}
	\label{delta.cuadrados.cerrado}
	If $m^2,n^2\in\mathbb{N}(\Delta)$, then $m^2n^2\in\mathbb{N}(\Delta)$. In other words, the set $\mathbb{N}(\Delta^2)$, under the usual multiplication, forms an abelian semigroup.
\end{corollary}

\begin{proof}
	It follows immediately from  \Cref{multiplos.cuadr.prop.delta}.
\end{proof}

With the help of a computer, one can verify that $30^{2}$ is the first element of $\mathbb{N}(\Delta^{2})$. Then $|\mathbb{N}(\Delta^{2})| = \infty$, by \Cref{multiplos.cuadr.prop.delta}.

The answer to (1) turns out to depend on whether $n$ is a perfect square or not. In the square case, uniqueness genuinely fails: given any two distinct $\Delta$-primitive squares, we can construct an $n$ admitting more than one decomposition of the form $n = \alpha^{2} m$ with $m$ $\Delta$-primitive.

We illustrate this with an example. With the help of a computer, one can determine that $84^{2}$ is the next smallest $\Delta$-primitive square after $30^{2}$. By \Cref{multiplos.cuadr.prop.delta},
$30^{2} a^{2}, 84^{2} b^{2} \in \mathbb{N}(\Delta)$ for all $a, b \in \mathbb{N}$. Therefore, if we find a pair $(a, b) \in \mathbb{N}^{2}$ satisfying
\begin{equation}
	\label{eq27}
	30^{2} a^{2} = 84^{2} b^{2},
\end{equation}
then $n = 30^{2} a^{2}$ is the number we are looking for. Obviously, \eqref{eq27} has the same solutions as $5a = 14b$ in $\mathbb{N}^{2}$. Since $5$ and $14$ are coprime, we obtain $a = 14 a_{1}$, $b = 5 b_{1}$ for some $a_{1}, b_{1} \in \mathbb{N}$, and substituting back yields $a_{1} = b_{1}$. Therefore, the solutions are of the form $(a, b) = (14k, 5k)$ for $k \in \mathbb{N}$, and $n = 420^{2} k^{2}$. The following proposition generalizes this procedure.

\begin{proposition}\label{prop:two-representations}
	If $m$ and $l$ are perfect squares, then there exist $a,b\in\mathbb{N}$ such that $a^2m=b^2l$. In particular, if $m$ and $l$ are $\Delta$-primitive, then $n=a^2m$ is a square with the $\Delta$ property that admits more than one representation as the product of a square and a $\Delta$-primitive.
\end{proposition}

\begin{proof}
	To prove the claim, it is necessary to find the integer solutions of the equation $a^2m=b^2l$ in the variables $a$ and $b$. Clearly, in $\mathbb{N}^2$, this is equivalent to solving
	\begin{equation}
		\label{eq26}
		am_1=bl_1,
	\end{equation}
	where $m_1=\sqrt{m}/\gcd(\sqrt{m},\sqrt{l})$ and $l_1=\sqrt{l}/\gcd(\sqrt{m},\sqrt{l})$. 
	
	Since $\gcd(m_1,l_1)=1$, it follows that $m_1|b$ and $l_1|a$. Then, $b=m_1b_1$ and $a=l_1a_1$ for some $a_1,b_1\in\mathbb{N}$. Substituting into \eqref{eq26}, we get $a_1=b_1$. Hence, the solution set of \eqref{eq26} is $\{ (a,b)\in\mathbb{N}^2: a=l_1k, \ b=m_1k, \ k\in\mathbb{N} \}$. \qedhere
\end{proof}

%%%%%%%%%%%%%%%%%%%%%%%%%%%%%%%%%%%%%%%%%%%%%%%%%%%%%%%%

\subsection{Uniqueness of decompositions in $\mathbb{N}(\Delta)$}

\Cref{prop:two-representations} shows that uniqueness of decomposition in $\mathbb{N}(\Delta)$ is essentially a non-square phenomenon. To make this idea precise, we now establish a partial uniqueness result in the non-square case, under the additional assumption that the two $\Delta$-primitive parts of the decompositions are coprime. The following three lemmas provide partial information towards the uniqueness of the decomposition of \Cref{n=alfa^2 m_m.delta.prim}.

\begin{lemma}
	\label{ru^2=sv^2}
	Let $r,s\in\mathbb{N}$ be square-free integers with $r\neq s$. Then the equation $ru^2 = sv^2$ has no solution in positive integers $u,v$.
\end{lemma}

\begin{proof}
	Suppose that $ru^2 = sv^2$ for some $u,v\in\mathbb{N}$. Then $\frac{r}{s} = \left(\frac{v}{u}\right)^2$ is a square in $\mathbb{Q}$. However, since $r$ and $s$ are distinct square-free integers, the rational number $r/s$ is not a square in $\mathbb{Q}$. This is a contradiction.
\end{proof}

Let $p$ be a prime number and let $n\in\mathbb{N}$ with $n\neq 0$. The \emph{$p$-adic valuation} of $n$, denoted by $v_p(n)$, is the largest integer $k\geq 0$ such that $p^k \mid n$. We recall the following basic property of $v_p$: for all $a,b\in\mathbb{N}$,
\[
v_p(ab) = v_p(a) + v_p(b),
\]
and $n$ is a perfect square if and only if $v_p(n)$ is even for every prime $p$.

\begin{lemma}
	\label{kx_ky.cuadrados.implica.x_y.cuadrados}
	Let $x,y\in\mathbb{N}$ be coprime, and suppose that $kx$ and $ky$ are perfect squares for some $k\in\mathbb{N}$. Then $x$, $y$, and $k$ are all perfect squares.
\end{lemma}

\begin{proof}
	Let $p$ be a prime. Since $kx$ is a square, we have $v_p(k) + v_p(x) \equiv 0 \pmod{2}$, and similarly, $v_p(k) + v_p(y) \equiv 0 \pmod{2}$. Hence, $v_p(x) \equiv v_p(y) \pmod{2}$. Since $\gcd(x,y)=1$, at most one of $v_p(x), v_p(y)$ is nonzero, hence both must be even. Therefore, $x$ and $y$ are perfect squares. It follows that $k$ is also a square.
\end{proof}

\begin{lemma}
	\label{a^2x=b^2y.implica.x_y.cuadrados}
	Let $a,b,x,y\in\mathbb{N}$ such that $\gcd(x,y)=1$ and $a^2 x = b^2 y$. Then, $x$ and $y$ are perfect squares.
\end{lemma}

\begin{proof}
	Since $x \mid b^2 y$ and $\gcd(x,y)=1$, it follows that $x \mid b^2$. Hence, there exists $k\in\mathbb{N}$ such that $b^2 = kx$. Similarly, $y \mid a^2$, so $a^2 = ky$. Thus, both $kx$ and $ky$ are perfect squares, and the result follows from \Cref{kx_ky.cuadrados.implica.x_y.cuadrados}.
\end{proof}

\begin{proposition}
	\label{prop:coprime-uniqueness}
	Let $n \in \mathbb{N}$ and suppose there exist $a, b \in \mathbb{N}$ and $x, y \in \mathbb{N}(\Delta)$ with  $a^2x=n=b^2y$, $\gcd(x,y)=1$ and $a<b$.	Then, $y < x$ and $\{n, x, y\} \subset \mathbb{N}(\Delta^{2})$. 
\end{proposition}

\begin{proof}
	By \Cref{a^2x=b^2y.implica.x_y.cuadrados}, both $x$ and $y$ are perfect squares. Since $n=a^2 x=b^2 y$, it follows that $n$ is also a perfect square. The claim about the $\Delta$ property follows from \Cref{multiplos.cuadr.prop.delta}. From $a<b$ follows immediately that $y<x$.
\end{proof}

The content of \Cref{prop:coprime-uniqueness} can be
paraphrased as follows: in the non-square case, no number $n$ admits two
decompositions $n = a^{2} x = b^{2} y$ with $x, y$ coprime
$\Delta$-primitives. 
%Together with
%\Cref{prop:two-representations}, which exhibits the
%non-uniqueness phenomenon in the square case, these results support
%the following conjecture, which answers~(Q3) in full for non-square $n$.
%
%\begin{conjecture}
%	\label{n.no-cuadr.implica.alpha,m.unicos}
%	Let $n=\alpha^2 m$, where $\alpha\in\mathbb{N}$ and $m$ is $\Delta$-primitive. If $n$ is not a perfect square, then $\alpha$ and $m$ are unique.
%\end{conjecture}
\Cref{prop:coprime-uniqueness} might suggest the conjecture that every non-square $n \in \mathbb{N}(\Delta)$ admits a
\emph{unique} decomposition $n = \alpha^{2} m$ with $\alpha \in \mathbb{N}$
and $m$ $\Delta$-primitive. This is false. The smallest counterexample is
\[
n \;=\; 5616 \;=\; 2^{4} \cdot 3^{3} \cdot 13 \;=\; 3^{2} \cdot 624 \;=\; 2^{2} \cdot 1404,
\]
where $624 = 2^{4} \cdot 3 \cdot 13$ is realized in $\mathbb{N}(\Delta)$
by the $\Delta$-triple $(8, 13, 13)$, and
$1404 = 2^{2} \cdot 3^{3} \cdot 13$ by $(26, 27, 36)$; both 624 and 1404 are $\Delta$-primitive. Note that they share the same square-free part, $3 \cdot 13 = 39$, and that $\gcd(624, 1404) = 156 > 1$, so the coprimality hypothesis of
\Cref{prop:coprime-uniqueness} is violated. The problem of
classifying the non-square integers $n$ admitting multiple
$\Delta$-primitive decompositions appears to be a delicate open
problem. 

The counterexample $5616$ points to a natural refinement of the
uniqueness question. For each square-free $s \in \mathbb{N}$, let
$P(s)$ denote the set of $\Delta$-primitives with square-free part $s$.
A non-square $n \in \mathbb{N}(\Delta)$ with square-free part $s$
admits multiple $\Delta$-primitive decompositions only if
$|P(s)| \geq 2$. We conjecture that the converse also holds: 

\begin{conjecture}
   If $|P(s)| = 1$, then every non-square $n \in \mathbb{N}(\Delta)$ with square-free part $s$ admits a unique $\Delta$-primitive decomposition.
\end{conjecture}

%%%%%%%%%%%%%%%%%%%%%%%%%%%%%%%%%%%%%%%%%%%%%%%%%%%%%%%%

\subsection{Infinitude of $\Delta$-primitive numbers}\label{subsec:delta_prim_are_infinite}

%\begin{lemma}
%	\label{Delta-num.y=z}
%	If $x$ is a natural number $\geq 2$, then
%	\[ n=x(2x-1)(3x-2)\in\mathbb{N}(\Delta). \]
%\end{lemma}
%
%\begin{proof}
%	...
%\end{proof}

\begin{theorem}
	\label{Delta-primitivos.son.infinitos}
	There are infinitely many $\Delta$-primitive numbers.
\end{theorem}

\begin{proof}
	Suppose, for contradiction, that there are only finitely many $\Delta$-primitive numbers, say $m_1 < m_2 < \cdots < m_k$. Choose an odd prime $p > m_k$. By \Cref{prop:max-duplicated}, $F(p) = p(2p-1)(3p-2) \in \mathbb{N}(\Delta)$. Since $F(p)>p>m_k$, $F(p)$ cannot be $\Delta$-primitive. By \Cref{n=alfa^2 m_m.delta.prim}, $F(p) = \alpha^2 m$ for some $\alpha \ge 2$ and some $\Delta$-primitive $m \in\{m_1, \ldots, m_k\}$. In particular, $m < p$, so $p \nmid m$. Moreover, $p \nmid (2p-1)$ and $p \nmid (3p-2)$. 
	
	Since $p \mid F(p) = \alpha^2 m$ and $p \nmid m$, we conclude that $p \mid \alpha^2$, and since $p$ is prime, we can write $\alpha=p\beta$, for some $\beta\in\mathbb{N}$. Hence,	$F(p)=p^2\beta^2m$, i.e., $(2p-1)(3p-2)=p\beta^2m$. Therefore, $p \mid (2p-1)(3p-2)$, a contradiction.
\end{proof}

%\begin{proof}
%	We will prove that there are infinitely many $\Delta$-primitives of the form $f(x)=x(2x-1)(3x-2)$, where $x$ is an integer $\geq 2$ (see \Cref{prop:max-duplicated}).
%	
%	Suppose that $P=\{f(x_i): i\in[k]\}$ is the set of all $\Delta$-primitives of this form, with $x_1<x_2<\ldots<x_k$. Then $f(x_1)<f(x_2)<\ldots<f(x_k)$, since $f$ is increasing on $[2,\infty)$. Now consider an (odd) prime $p>f(x_k)$. Since $f(p)>f(x_k)$, the number $f(p)$ cannot be $\Delta$-primitive, i.e.,
%	\begin{equation}
%		f(p)=\alpha^2f(n), 
%		\label{eq1}
%	\end{equation}
%	for some $\alpha\geq 2$, $n\in\{x_1,\ldots, x_k\}$. Then, \eqref{eq1} becomes
%	\begin{equation}
%		p(2p-1)(3p-2)=\alpha^2n(2n-1)(3n-2).
%		\label{eq2}
%	\end{equation} 
%	We see that $p|\alpha^2f(n)$, but clearly $p$ cannot divide $n$, $2n-1$, or $3n-2$. Consequently, $p|\alpha$, so we can write $\alpha=p\beta$ for some $\beta\in\mathbb{N}$. Substituting into \eqref{eq2} yields
%	\begin{equation}
%		(2p-1)(3p-2)=p\beta^2f(n),
%		\label{eq3}
%	\end{equation} 
%	which is absurd since the right-hand side of \eqref{eq3} is divisible by $p$, while the left-hand side is not.  
%\end{proof}

\begin{conjecture}
	\label{conj:inf_many_delta-squares}
	There are infinitely many $\Delta$-primitive squares.
\end{conjecture}

%%%%%%%%%%%%%%%%%%%%%%%%%%%%%%%%%%%%%%%%%%%%%%%%%%%%%%%%%%%

\section{Generating polynomials} \label{sec:polinom.generadores}

In this section, we revisit and expand upon an idea that arose in \Cref{prop:max-duplicated,prop:max-generic}, developing a method for creating generating polynomials for numbers with the $\Delta$ property. \\

If $x$ is a natural number $>1$, consider
\begin{equation}
	\label{eq34}
	n=n(x)=x(x+2)(2x+1).
\end{equation}
Now observe that
\[ x(2x+1)-(x+2) = 2(x(x+2)-(2x+1)), \]
which is equivalent to
\begin{equation}
	\label{eq29}
	\frac{n}{x+2}-(x+2)=2\Big( \frac{n}{2x+1}-(2x+1) \Big).
\end{equation}
Since for all $x>1$ we have that $x+2, 2x+1\in D_n$ and
\[ 1<x+2<2x+1<\sqrt{n}, \]
it follows from \eqref{eq29} that $n\in\mathbb{N}(\Delta)$ for all $x>1$, since all requirements of \Cref{Delta.prop.caracterizacion} are satisfied. These simple calculations, together with what we saw in \Cref{prop:max-duplicated,prop:max-generic}, suggest the possibility of finding a more general method for constructing generating polynomials for numbers with the $\Delta$ property. 

%%%%%%%%%%%%%%%%%%%%%%%%%%%%%%%%%%%%%%%%%%%%%%%%%%%%%%%%%%%%%%

\subsection{The duplicated method ``$y=z$''}

Consider polynomials of the form
\[ n(x)=\prod_{i=1}^{k}(a_ix+b_i), \]
with $a_i,b_i\in\mathbb{Z}$, such that each linear factor $a_ix+b_i$ of $n$ is an element of $D_n$. Clearly, we must have $k\geq 2$, since we always need $n$ to have at least two nontrivial divisors for sufficiently large $x$. 

If $k=2$, then suppose that $n=(a_1x+b_1)(a_2x+b_2)$ satisfies the $\Delta$ property via its divisors $a_1x+b_1$ and $a_2x+b_2$. Without loss of generality, this can only happen in the following way:
\begin{equation}
	\label{eq30}
	\frac{n}{a_1x+b_1}-(a_2x+b_2)=2\Big( \frac{n}{a_2x+b_2}-(a_1x+b_1) \Big).
\end{equation}
Simplifying expression \eqref{eq30}, we obtain
\begin{equation}
	\label{eq31}
	(a_2-a_1)x+(b_2-b_1)=0.
\end{equation}
The key to the method we wish to present is this: if the left-hand side of \eqref{eq31} is the identically zero polynomial, then \eqref{eq30} holds. Therefore, we take $a_2-a_1=0=b_2-b_1$, that is, $a_2=a_1$ and $b_2=b_1$. However, now $n=(a_1x+b_1)^2$, so $a_1x+b_1=\sqrt{n}$, which is not acceptable by  \Cref{Delta.prop.caracterizacion}. Hence, we must have $k\geq 3$. 

For each $k\geq 3$, we now define the following polynomial:
\begin{equation}
	\label{eq32}
	f(x) = \frac{n}{a_1x + b_1} - (a_1x + b_1) - 2\left(\frac{n}{a_2x + b_2} - (a_2x + b_2)\right).
\end{equation}
Let $k=3$. Observe that the conditions $0<a_1<a_2$ and $a_3>0$ are sufficient to guarantee the existence of some $n_0\in\mathbb{N}$ such that the inequalities
\begin{equation*}
	1<a_1x+b_1<a_2x+b_2<\sqrt{n(x)}
\end{equation*}
hold simultaneously for $x\in[n_0,\infty)$. Moreover, since all $a_i$ are positive, we also ensure that $\deg(n)=3$ and that $n$ is eventually increasing. Through basic manipulations, we rewrite \eqref{eq32} as $f(x)=c_2x^2+c_1x+c_0$, where
\[ c_2= a_2a_3 - 2a_1a_3, \]
\[ c_1= a_2b_3 + a_3b_2 - a_1 - 2a_1b_3 - 2a_3b_1 + 2a_2, \]
\[ c_0= b_2b_3 - b_1 - 2b_1b_3 + 2b_2. \]
Since $f=0$ implies $n\in\mathbb{N}(\Delta)$ for all $x\geq n_0$, we must find integer solutions of the system $c_2=c_1=c_0=0$ in the six unknowns $a_i,b_i$. From $c_2=0$ we deduce that $a_2=2a_1$, which is compatible with assuming $a_1<a_2$. Substituting $a_2=2a_1$ into $c_1=0$ and $c_0=0$, we obtain respectively that $(2b_1-b_2)a_3=3a_1$ and $(2b_1-b_2)b_3=2b_2-b_1$. Since $a_1,a_3\neq 0$, it follows that $2b_1-b_2\neq 0$, which allows us to isolate:
\begin{equation*}
	a_3=\frac{3a_1}{2b_1-b_2}, \quad b_3=\frac{2b_2-b_1}{2b_1-b_2}.
\end{equation*}
We note that $a_1,a_3>0$ implies $2b_1>b_2$; in particular, $(b_1,b_2)\neq(0,0)$ and $b_3\geq 0$ if and only if $2b_2\geq b_1$. We see that $a_1,b_1$, and $b_2$ end up being free parameters. Finally, for $a_3,b_3\in\mathbb{Z}$, it is sufficient that $2b_1-b_2$ divides $\gcd(3a_1,2b_2-b_1)$. We summarize all this in the following theorem.

\begin{theorem}
	\label{polinom.Delta-generadores_deg=3}
	Let $a,b,c\in\mathbb{Z}$ such that $a>0$ and $2b>c$. Define
	\[ \alpha=\frac{3a}{2b-c}, \quad \beta=\frac{2c-b}{2b-c}, \]
	and consider the polynomial $n(x)=(ax+b)(2ax+c)(\alpha x+\beta)$. Let $n_0$ be the smallest natural number such that the inequalities
	\[ 1<ax+b<2ax+c<\sqrt{n(x)} \]
	hold simultaneously for $x\in[n_0,\infty)$. If $x\in\mathbb{N}$ and
	\[ (2b-c)\mid\gcd(3a,2c-b), \]
	then $n(x)\in\mathbb{N}(\Delta)$ for all $x\geq n_0$.
\end{theorem}

\begin{proof}
	It follows from the preceding discussion.
\end{proof}

Let us see some examples of how to use  \Cref{polinom.Delta-generadores_deg=3} to construct generating polynomials for numbers with the $\Delta$ property.

If $\gcd(3a,2c-b)$ is a prime $p$, then $2b-c\in\{1,p\}$. If $2b-c=1$, then $c=2b-1$, $\alpha=3a$, and $\beta=3b-2$, yielding
\[ n(x)=(ax+b)(2ax+2b-1)(3ax+3b-2), \]
for $a\in\mathbb{N}$ and $b\in\mathbb{Z}$. For example, if $(a,b)=(3,4)$, then
\[ (3x+4)(6x+7)(9x+10)\in\mathbb{N}(\Delta), \]
for all $x\in\mathbb{N}$. If instead $(a,b)=(1,0)$, we recover the polynomial $x(2x-1)(3x-2)$ from  \Cref{prop:max-duplicated}.

If $2b-c=p$, then $c=2b-p$, $\alpha=\frac{3a}{p}$, and $\beta=\frac{3b}{p}-2$. Then: either $p=3$, or $p\neq 3$ and $p\mid\gcd(a,b)$. If $p=3$, then $c=2b-3$, $\alpha=a$, and $\beta=b-2$, so
\[ n(x)=(ax+b)(2ax+2b-3)(ax+b-2). \]
For example, if $(a,b)=(1,-1)$, then
\[ (x-1)(2x-5)(x-3)\in\mathbb{N}(\Delta), \]
for all $x\geq 5$, $x\in\mathbb{N}$. If instead $(a,b)=(1,2)$, we recover the polynomial \eqref{eq34}. If $p\neq 3$ and $(a,b)=(2,-4)$, then $p=2$, $c=-10$, $\alpha=3$, and $\beta=-8$. Hence,
\[ 4(x-2)(2x-5)(3x-8)\in\mathbb{N}(\Delta), \]
for all $x\geq 4$, $x\in\mathbb{N}$.

%%%%%%%%%%%%%%%%%%%%%%%%%%%%%%%%%%%%%%%%%%%%%%%%%%%%%%%%

\subsection{The failure of the duplicated method for $k \geq 4$}
\label{subsec:failure}

The method developed above relies crucially on the \emph{duplicated regime} $y = z$ of the $\Delta$ condition: two of the three divisors in a $\Delta$-triple are forced to coincide. At the level of the generating polynomial, this manifests itself in the factor $2$ appearing in \eqref{eq32}, which couples two linear factors of $n(x)$ as $x$ and $y = z$. One may ask whether the same duplicated method can be extended to products of $k \geq 4$ linear factors in order to produce generating polynomials of higher degree. The purpose of this subsection is to show that the answer is negative: the duplicated method breaks down for every $k \geq 4$, by an obstruction concentrated on the leading coefficients of the resulting polynomial.

Concretely, fix $k \geq 3$ and consider a product
\[
n(x) \;=\; \prod_{i=1}^{k}(a_{i} x + b_{i}), \qquad a_{i} \in \mathbb{N},\ b_{i} \in \mathbb{Z},
\]
in which $X = a_{1} x + b_{1}$ and $Y = a_{2} x + b_{2}$ are
chosen to play the role of $x$ and $y = z$ in the $\Delta$ condition.
Write $M$ for the product of the remaining $k-2$ linear factors,
so that $n/X = Y  M$ and $n/Y = X  M$. The duplicated
$\Delta$ condition $\tfrac{n}{X} - X = 2\!\left(\tfrac{n}{Y} - Y\right)$
becomes the polynomial identity $f= 0$, where
\begin{equation}\label{eq:duplicated-fk}
	f \;=\; (Y - 2X)M \;+\; (2Y - X).
\end{equation}

\begin{theorem}
	\label{thm:failure-kgeq4}
	Let $k \geq 4$. There is no choice of coefficients $a_{i}, b_{i} \in \mathbb{Z}$ with $a_{i}>0$ for which the identity
	$f= 0$ (see \eqref{eq:duplicated-fk}) holds.
\end{theorem}

%\begin{proof}
%	Since $\deg (M) = k - 2$ and $\deg(Y - 2X) \leq 1$, the first summand
%	of~\eqref{eq:duplicated-fk} has degree at most $k-1$, while the second
%	summand has degree $1$. As $k \geq 4$, the two summands have disjoint
%	leading terms, and the identity $f= 0$ decomposes into the
%	vanishing of the leading coefficients of $M  (Y - 2X)$ at degrees
%	$x^{k-1}$ and $x^{k-2}$, together with lower-order conditions.
%	
%	The leading coefficient of $M  (Y - 2X)$ is $(a_{j} - 2 a_{i})\prod_{l \neq i, j} a_{l}$. Since $a_{l} > 0$ for all $l$, the vanishing of this coefficient forces $a_{j} =2 a_{i}$. Substituting it into $Y - 2X$, this factor reduces to the constant $b_{j} - 2 b_{i}$, and hence the coefficient of $x^{k-2}$ in $M  (Y - 2X)$ is $(b_{j} - 2 b_{i})\prod_{l \neq i, j} a_{l}$. Again using $a_{l} > 0$ for all $l$, its vanishing forces $b_{j} = 2 b_{i}$.
%
%	Combining $a_{j} =2 a_{i}$ and $b_{j} = 2 b_{i}$ yields $Y = 2X$ as polynomials in $x$, so the first summand of~\eqref{eq:duplicated-fk} vanishes identically. The identity $f =0$ then reduces to $0=2Y-X=3(a_{i} x + b_{i})$, which is impossible since $a_{i} > 0$.
%\end{proof}

\begin{proof}
	Since $\deg(M) = k - 2$ and $\deg(Y - 2X) \leq 1$, the first summand of~\eqref{eq:duplicated-fk} has degree $\le k-1$, whereas the second summand $2Y - X$ has degree  $\le 1$. Since $k-1 \geq 3$, the higher-degree terms of the first summand cannot possibly be canceled by the second summand. For the identity $f=0$ to hold, the coefficients of $x^{j}$ for $j\ge 2$ in $(Y - 2X)M$ must vanish, since $\deg(2Y-X)\le 1$.
	
	The coefficient of $x^{k-1}$ in $(Y - 2X)M$ is $(a_2 - 2a_1) \prod_{i>2} a_i$. Since $a_i > 0$ for all $i$, forcing this to zero requires $a_2 = 2a_1$. This substitution cancels the $x$ term inside $Y - 2X$, reducing it to the constant $b_2 - 2b_1$. Consequently, the maximum degree of the first summand drops to $k-2$. However, since $k \geq 4$, this degree is at least $2$, which still strictly exceeds the degree of the second summand. Therefore, the coefficient $(b_2 - 2b_1) \prod_{i >2} a_i$ of $x^{k-2}$ must also vanish. Hence, $b_2 = 2b_1$.
	
	Combining $a_2 = 2a_1$ and $b_2 = 2b_1$ implies that $Y = 2X$ as polynomials. Substituting this back into~\eqref{eq:duplicated-fk}, the identity $0=f$ then collapses to: $0 = 0M+(2(2X) - X) =3(a_1 x + b_1)$. This is a clear contradiction, as $3a_1$ cannot be zero given that $a_1 > 0$.
\end{proof}

%The proof exposes the structural reason behind the failure: the duplicated method \eqref{eq:duplicated-fk} requires a single linear factor $Y - 2X$ to compensate for a residual factor $M$ of degree $k-2$. For $k = 3$, $M$ is constant and this compensation is feasible; for $k \geq 4$, $M$ has degree $\geq 2$, and the compensation is simply not available within the duplicated ansatz.

\Cref{thm:failure-kgeq4} should not be read as an obstruction to generating polynomials of degree $\geq 4$, but rather as an indication that a different ansatz is required. Indeed, for $k \geq 4$ the natural replacement is the \emph{generic regime} $y < z$: three distinct linear factors $X, Y, Z$ of $n(x)$ play the roles of the three divisors in a $\Delta$-triple for $n$ and the $k-3$ remaining factors form the free mass of the construction. The resulting polynomial identity reduces to a system of equations whose systematic treatment we leave to future work.

%Combined with \Cref{prop:polynomials}, \Cref{thm:failure-kgeq4} establishes a sharp dichotomy
%between two regimes in which the $\Delta$ condition can be realized through linear families:
%	\begin{itemize}
%		\item[(i)] \emph{Duplicated regime} ($y = z$): corresponds to degree~$3$
%		generating polynomials. In this regime, every duplicated $\Delta$-triple $(x, y, y)$ satisfies the universal linear relation
%		$y = 2x - c$, where $c = 2x - y$ is a positive integer.
%		\item[(ii)] \emph{Generic regime} ($y < z$): corresponds to generating
%		polynomials of degree $\geq 4$, in which the three divisors $x, y, z$
%		are realized by three distinct linear factors of $n(x)$.
%	\end{itemize}
%In particular, no linear family of degree $\geq 4$ can be reached via the duplicated ansatz, and no linear family of degree $3$ can be reached via the generic ansatz. This dichotomy explains the structural asymmetry observed between \Cref{prop:cubicgenerator} and \Cref{prop:polynomials}.

%%%%%%%%%%%%%%%%%%%%%%%%%%%%%%%%%%%%%%%%%%%%%%%%%%%%%%%%

\subsection{The failure of the generic method for $k = 3$}
\label{subsec:failure-generic-k3}

The previous subsection rules out the duplicated ansatz $y = z$ for $k \geq 4$. We now establish the complementary obstruction: in the generic regime $y < z$, no linear $\Delta$-family exists for $k = 3$. The proof reduces \eqref{ecuacion2.Delta.prop} to a polynomial identity governed by the polynomial $x^{2} - x + 1$, whose irreducibility over $\mathbb{R}$ provides the obstruction. The starting point is the following result, characterizing the case $n = xyz$.

\begin{proposition}\label{prop:n=xyz}
	Let $(x, y, z)$ be a $\Delta$-triple for $n$. Then, $n = xyz$ if and only if
	\begin{equation}\label{eq:xyz-identity}
		(y - x + 1)\,(z - x + 1) \;=\; x^{2} - x + 1.
	\end{equation}
\end{proposition}

\begin{proof}
	Assume $n = xyz$. Substituting into \eqref{ecuacion2.Delta.prop} gives 
	\begin{equation}\label{eq:xyz-pre-identity}
		yz - xy - xz \;=\; x - y - z.
	\end{equation}
	Adding $(y - x) + (z - x) + 1$ to both sides and grouping yields $(y-x)(z-x) + (y-x) + (z-x) + 1 = x^{2} - x + 1$, which factors as \eqref{eq:xyz-identity}. The converse is the same manipulation in reverse: \eqref{eq:xyz-identity} expands to \eqref{eq:xyz-pre-identity}, which together with $(x, y, z) \in D_n^3$ forces $n = xyz$ via \eqref{ecuacion2.Delta.prop}.
\end{proof}

The smallest examples where \Cref{prop:n=xyz} applies are 385 and 2080, with $\Delta$-triples $(5,7,11)$ and $(8,10,26)$, respectively.

\begin{theorem}
	\label{thm:failure-generic-k3}
	There do not exist polynomials $X, Y, Z \in \mathbb{Z}[t]$ of degree $\leq 1$, such that $\bigl(X(t),\, Y(t),\, Z(t)\bigr)$ is a $\Delta$-triple for $n(t) = (XYZ)(t)$ for infinitely many $t \in \mathbb{N}$.
\end{theorem}

\begin{proof}
	We proceed by contradiction. Suppose there exist polynomials $X, Y, Z$ $\in\mathbb{Z}[t]$ of degree $\le 1$ such that the set 
	\[ S=\{t\in\mathbb{N}: (X(t),Y(t),Z(t)) \ \text{is a $\Delta$-triple for} \ n(t) \} \]
	is infinite. By \Cref{prop:n=xyz}, the integer identity
	\[
	\bigl(Y(t) - X(t) + 1\bigr)\bigl(Z(t) - X(t) + 1\bigr) \;=\; X(t)^{2} - X(t) + 1
	\]
	holds for all $t \in S$. Two polynomials in $\mathbb{Z}[t]$ that agree on an infinite set are equal; hence
	\begin{equation}\label{eq:identity-X-Y-Z}
		\bigl(Y - X + 1\bigr)\bigl(Z - X + 1\bigr) \;=\; X^{2} - X + 1
	\end{equation}
	in $\mathbb{Z}[t]$. Write $X = at + b$, for some $a,b\in\mathbb{Z}$. If $a=0$ and $b\le 1$, then $X(t)$ would not be a component of a $\Delta$-triple for all $t$. If $a=0$ and $b>1$, then $X(t)$ would be a fixed component of infinitely many $\Delta$-triples, contradicting \Cref{cor:finite-fixed-component}. Thus, $a\neq 0$. The right-hand side of \eqref{eq:identity-X-Y-Z} equals $a^{2}t^{2}+a(2b - 1)t+(b^{2} - b + 1)$, whose discriminant is $-3a^{2}$. Hence, the right-hand side  of \eqref{eq:identity-X-Y-Z} is irreducible over $\mathbb{R}[t]$. The left-hand side of \eqref{eq:identity-X-Y-Z}, however, is the product of two polynomials of degree $\leq 1$ in $\mathbb{R}[t]$, contradicting irreducibility.
\end{proof}

%In contrast to the duplicated regime, where a single algebraic identity in three integer parameters generates the infinite linear family of \Cref{polinom.Delta-generadores_deg=3}, the integer solutions of \eqref{eq:xyz-identity} are intrinsically sporadic: each is dictated by a particular nontrivial factorization of $x^{2} - x + 1$ for some specific $x$, and they cannot be threaded into a one-parameter linear family. The smallest examples are
%\[
%(x, y, z) = (5, 7, 11), \quad x^{2} - x + 1 = 21 = 3 \cdot 7, \quad n = 385,
%\]
%and
%\[
%(x, y, z) = (8, 10, 26), \quad x^{2} - x + 1 = 57 = 3 \cdot 19, \quad n = 2080.
%\]
%Realizations of \eqref{eq:xyz-identity} with $\gcd(x, y) > 1$ or $\gcd(x, z) > 1$ -- such as $(8, 10, 26)$ -- lie strictly outside the prime case treated in \Cref{pqr.prop.Delta}.

%%%%%%%%%%%%%%%%%%%%%%%%%%%%%%%%%%%%%%%%%%%%%%%%%%%%%%%%

\subsection{Elliptic curves and $\Delta$-primitive squares}
\label{subsec:elliptic}

The generating polynomials of \Cref{polinom.Delta-generadores_deg=3} produce elements of $\mathbb{N}(\Delta)$ in abundance, but they say nothing about which of those elements are squares. This is precisely the question left open by Conjecture \ref{conj:inf_many_delta-squares}, and it admits a natural reformulation in terms of integer points on elliptic curves.

Let $f\in\mathbb{Z}[x]$ be any of the cubic generating polynomials supplied by \Cref{polinom.Delta-generadores_deg=3}. Then $f(x_0)$ is a perfect square (hence, a square in $\mathbb{N}(\Delta)$) precisely when $(x_0,y_0)$ is an integer point on the affine cubic $E_f: y^2=f(x)$. By Siegel's theorem~\cite{siegel1929}, each individual curve $E_f$ carries only finitely many integer points. However, \Cref{polinom.Delta-generadores_deg=3} provides an infinite family $\{f_\lambda\}$ of cubic generating polynomials, and therefore an infinite family $\{E_{f_\lambda}\}$ of elliptic curves. This suggests a concrete strategy toward Conjecture \ref{conj:inf_many_delta-squares}.

\begin{problem}
	\label{prob:elliptic-families}
	Can one select a family $\{f_\lambda\}_{\lambda\in\Lambda}$ of cubic generating polynomials such that the integer points of the associated elliptic curves $\{E_{f_\lambda}\}_{\lambda\in\Lambda}$ produce infinitely many \emph{distinct} $\Delta$-primitive squares?
\end{problem}

A positive answer to Problem \ref{prob:elliptic-families} would settle Conjecture \ref{conj:inf_many_delta-squares} affirmatively and would constitute, to our knowledge, the first application of diophantine geometry to 2-switch-degree theory. 

%%%%%%%%%%%%%%%%%%%%%%%%%%%%%%%%%%%%%%%%%%%%%%%%%%%%%%%%%%%%%%

\section{Numbers without the $\Delta$ property}\label{sec:Num.sin.peop.Delta}

In this section we present a long list of families of numbers that do not satisfy the $\Delta$ condition. This class of numbers is infinite. 
%In fact, it states that no even integer not divisible by 4 has the $\Delta$ property. A key insight is that $pk \notin \mathbb{N}(\Delta)$ when $p$ is a prime $\ge k$. A similar phenomenon occurs for numbers $n$ of the form $p^x q^y$ ($p, q$ primes): if $q > p^x$, then $n \notin \mathbb{N}(\Delta)$.
We note that numbers with very few prime factors in their factorization are less likely to have the $\Delta$ property. 
%For example, powers of a prime or products of certain powers of two primes, such as $pq, pq^2, p^2 q^2$, and $p^k q$. We conjecture that if $n$ is a $\Delta$-primitive of the form $p^x q^y$ ($p, q$ primes), then $n \in \{24, 40\}$, and we characterize when a product of three primes $pqr \in \mathbb{N}(\Delta)$. 
Moreover, having information about the ordering of the elements of $D_n$ is crucial to prove that $n \notin \mathbb{N}(\Delta)$. After all this analysis, we return to split graphs, establishing a result about the $n$-simple induced cycles of $\Phi$, when $n$ does not satisfy the $\Delta$ condition and is not a square.

\begin{theorem}
	\label{2impar.no.tiene.prop.Delta}
	If $k$ is odd, then $2k \notin \mathbb{N}(\Delta)$.
\end{theorem}

\begin{proof}
	If $2k = ab$, then $a$ and $b$ cannot have the same parity, since otherwise $2k$ would be odd or a multiple of 4. Consequently, each element of $D^*_{2k}$ is odd, making every member of $D^+_{2k}$ even. Thus, $D^*_{2k} \cap D^+_{2k} = \varnothing$.  
\end{proof}

 \Cref{2impar.no.tiene.prop.Delta} has two direct consequences worth highlighting. The first is that $|\mathbb{N} - \mathbb{N}(\Delta)| = \infty$. The second is that $\mathbb{N}(\Delta)$ contains no even square-free numbers. In other words, every even number in $\mathbb{N}(\Delta)$ is a multiple of 4.
 
 %%%%%%%%%%%%%%%%%%%%%%%%%%%%%%%%%%%%%%%%%%%%%%%%%%%%%%%%%%%%%
 
 \subsection{The dominating-prime obstruction}

The next result is of notable importance because it tells us that a number cannot satisfy the $\Delta$ condition if it is divisible by a ``sufficiently large" prime.

\begin{theorem}
	\label{p>2k-2_entonces.pk.no tiene.prop.Delta}
	If $k\in\mathbb{N}$ and $p$ is a prime $\geq k$, then $pk\notin\mathbb{N}(\Delta)$.
\end{theorem}

\begin{proof}
	Suppose $k\in\mathbb{N}$, $p$ is a prime $\geq k$, but $pk\in\mathbb{N}(\Delta)$. Then, by  \Cref{Delta.prop.caracterizacion}, there exist $x, y, z\in D_{pk}=D_k\cup (pD_k)$ such that $2\leq x<y\leq z<\sqrt{pk}$ and
	\begin{equation}
		\label{eqpk1}
		\frac{pk}{x}-x=\frac{pk}{y}-y+\frac{pk}{z}-z.
	\end{equation}
	Since $p\geq k$, any element of $pD_k$ is at least $p\geq\sqrt{pk}$. Hence, $x, y, z\in D_k$. We rewrite \eqref{eqpk1} as
	\begin{equation}
		\label{eqpk2}
		p\left(\frac{k}{z}+\frac{k}{y}-\frac{k}{x}\right)=y+z-x.
	\end{equation}
	Since the left-hand side of \eqref{eqpk2} is a product of two integers, it follows that $p\mid(y+z-x)$. Therefore,
	\[ 0<y+z-x\leq k+(k-2)<2k\leq 2p, \]
	which forces $y+z-x=p$. Using \eqref{ecuacion2.Delta.prop} we get
	\begin{equation}
		\label{eqpk3}
		k(xy+xz-yz)=xyz.
	\end{equation}
	The hypothesis $p\geq k$, combined with $p=y+z-x$ and \eqref{eqpk3}, yields
	\[ (y+z-x)(xy+xz-yz)\geq xyz, \]
	which after elementary manipulations reduces to
	\[ (y+z)(z-x)(x-y)\geq 0. \]
	Since $y+z$ and $z-x$ are positive, we finally get $x\ge y$, contradicting the hypothesis $x<y$. Hence, $pk\notin\mathbb{N}(\Delta)$.
\end{proof}

%\begin{theorem}
%	\label{p>2k-2_entonces.pk.no tiene.prop.Delta}
%	If $k \in \mathbb{N}$ and $p$ is a prime $\geq 2k - 1$, then $pk \notin \mathbb{N}(\Delta)$.
%\end{theorem}
%
%\begin{proof}
%	Suppose $k \in \mathbb{N}$, $p$ is a prime $\geq 2k - 1$, but $pk \in \mathbb{N}(\Delta)$. Then, by  \Cref{Delta.prop.caracterizacion}, there exists a triple $T = \{x, y, z\} \subset D_{pk} = D_k \cup p D_k$ such that $2 \leq x < y \leq z < \sqrt{pk}$ and 
%	\begin{equation}
%		\label{eq19}
%		\frac{pk}{x} - x = \frac{pk}{y} - y + \frac{pk}{z} - z.
%	\end{equation}
%	Since $k \leq 2k - 1 \leq p$, we have $\sqrt{pk} \leq \sqrt{p^2} = p$. Thus, $x, y, z < p$, which shows that $T \subset D_k$. We rewrite \eqref{eq19} as:
%	\begin{equation}
%		\label{eq20}
%		p \left( \frac{k}{z} + \frac{k}{y} - \frac{k}{x} \right) = z + y - x
%	\end{equation}
%	Since the left-hand side of \eqref{eq20} is a product of two integers, it follows that $p | (z + y - x)$. Therefore, 
%	\[ 2k - 1 \leq p \leq z + (y - x) \leq k + (k - 2) < 2k - 1, \]
%	which is absurd. Hence, $pk \notin \mathbb{N}(\Delta)$.
%\end{proof}

 \Cref{p>2k-2_entonces.pk.no tiene.prop.Delta} is useful for generating numbers without the $\Delta$ property, since for each $k \in \mathbb{N}$, we have infinitely many choices for $p$. It is interesting to note what the contrapositive of this proposition tells us: if $pk \in \mathbb{N}(\Delta)$, then $p < k$. In other words, there are only finitely many prime multiples of a number that satisfy the $\Delta$ property. 
Another important observation is that \Cref{p>2k-2_entonces.pk.no tiene.prop.Delta} implicitly provides a recipe to construct, for each prime $k_1$, an infinite sequence $\ell(k_1)=(k_i)_{i\in\mathbb{N}}$ of square-free numbers that do not have the $\Delta$ property: set $k_{i+1}=p_i k_i$, where $p_i$ is any prime $\geq k_i$ coprime to $k_i$. For instance, starting from $k_1=2$ with successive choices $p_i=3,7,43,\ldots$, we obtain $\ell(2)=(2,6,42,1806,\ldots)$.
%Another important observation is that  \Cref{p>2k-2_entonces.pk.no tiene.prop.Delta} implicitly provides a recipe to construct, for each prime $k_1$, an infinite sequence $\ell(k_1) = (k_i)_{i \in \mathbb{N}}$ of square-free numbers that do not have the $\Delta$ property. Let's see how to do this with an example. We want to determine the first terms of $\ell(2)$. Since $k_1 = 2$, let $p_1 = 3 \geq 2k_1 - 1$. Then, $k_2 = p_1 k_1 = 6$. If $p_2 = 11 \geq 2k_2 - 1$, then $k_3 = p_2 k_2 = 66$. If $p_3 = 131 \geq 2k_3 - 1$, then $k_4 = p_3 k_3 = 8646$. And so on. In general, $k_{i+1} = p_i k_i$, where $p_i$ is a prime $\geq 2k_i - 1$.

%We now prove that no power of a prime has the $\Delta$ property. 
%A preliminary step is to verify that this claim holds for powers of order 1 and 2.

%\begin{lemma}
%	\label{p_p^2.no.tienen.prop.Delta}
%	If $k \in \{1, 2\}$ and $p$ is prime, then $p^k \notin \mathbb{N}(\Delta)$.
%\end{lemma}
%
%\begin{proof}
%	If $k = 1$, then $D^*_p = \{p - 1\}$ and $D^+_p = \{2(p - 1)\}$ are clearly disjoint. For $k = 2$, let $n = p^2$. Then, $D_n^* = \{n - 1, 0\}$ and $D_n^+ = \{2(n - 1)\}$. Hence, it is again evident that $D_n^* \cap D_n^+ = \varnothing$.
%\end{proof}
%
%As a mere curiosity, note that it can be proven that no prime satisfies the $\Delta$ condition simply by taking $k = 1$ in  \Cref{p>2k-2_entonces.pk.no tiene.prop.Delta}.  

The next result we prove, closely resembles  \Cref{p>2k-2_entonces.pk.no tiene.prop.Delta} in style. It states that a number cannot have the $\Delta$ property if its factorization contains exactly two primes and one of them is ``much larger" than the other.

\begin{theorem}
	\label{p^xq^y.no.tiene.prop.Delta.q>p^x}
	Let $x, y \in \mathbb{N}$, with $y \geq 2$, and let $p, q$ be distinct primes. If $n = p^x q^y$ and $q > p^x$, then $n \notin \mathbb{N}(\Delta)$.
\end{theorem}

\begin{proof}
	We proceed by contradiction, first assuming that $n$ is $\Delta$-primitive. Then, there exist non-negative integers $a, c, e \leq x$ and $b, d, f \leq y$ such that 
	\begin{equation}
		\frac{n}{p^a q^b} - p^a q^b = \frac{n}{p^c q^d} - p^c q^d + \frac{n}{p^e q^f} - p^e q^f,
		\label{eq11}
	\end{equation}  
	where 
	\begin{equation*}
		1 < p^a q^b < p^c q^d \leq p^e q^f \leq \max \{ \alpha \in D_n : \alpha < \sqrt{n} \} = p^z q^{\lfloor y/2 \rfloor},
	\end{equation*}
	for some $z \leq x$. Note that $f \leq \lfloor y/2 \rfloor$. Indeed, if $f > \lfloor y/2 \rfloor$, we would have $p^e q^{f - \lfloor y/2 \rfloor} \leq p^z < q$, which is absurd. 
	
	We claim that $b \leq d \leq f$. Suppose that $b > d$. Since $p^a q^b < p^c q^d$, we have $p^{c - a} > q^{b - d}$. However, $a, c \leq x$ implies $p^{c - a} \leq p^x < q \leq q^{b - d}$, a contradiction. Similarly, if $d > f$, then $p^c q^d \leq p^e q^f$ would force $p^{e - c} \geq q^{d - f}$, but $p^{e - c} \leq p^x < q \leq q^{d - f}$, a contradiction. Hence $d \leq f$.
	%Using the same argument, we can also see that $b \leq d \leq f$. 
	
	Furthermore, since $\lfloor y/2 \rfloor \leq y/2 < y$, it follows that $y - b, y - d$, and $y - f$ are all positive. Now, we rewrite equality \eqref{eq11} as
	\begin{equation}
		p^{x - a} q^{y - b} - p^a q^b = p^{x - c} q^{y - d} - p^c q^d + p^{x - e} q^{y - f} - p^e q^f.
		\label{eq12}
	\end{equation} 
	
	If $b, d, f > 0$, then we can divide both sides of \eqref{eq12} by $q$, obtaining that $m = p^x q^{y - 2} \in \mathbb{N}(\Delta)$ (recall that $y - 2 \geq 0$ by hypothesis). Since $n = q^2 m$ is $\Delta$-primitive, it follows that $q = 1$, a contradiction. Hence, $0 \in \{b, d, f\}$. 
	
	If $f = 0$, then $b = d = 0$ and \eqref{eq12} becomes
	\[ q^y (p^{x - a} - p^{x - c} - p^{x - e}) = p^a - p^c - p^e. \]
	Thus, $q^y$ divides $|p^a - p^c - p^e|$. Since $p^a q^0 < p^c q^0$, it follows that $|p^a - p^c - p^e| = (p^c - p^a) + p^e \leq p^c + p^e \leq 2q$. But then $q^y \leq 2q$, which is absurd. Therefore, it must be that $f > 0$.
	
	If $d = 0$, then $b = 0$ as well. Since $p^a q^0 < p^c q^0$, it follows that $a < c$ and, thus, $x - c < x - a$, so $p^{x - c} < p^{x - a}$. With these observations in mind, we rewrite \eqref{eq12} appropriately, obtaining the following contradiction:
	\[ 0 < q^{y - f} [q^f (p^{x - a} - p^{x - c}) - p^{x - e}] = -p^e q^f - (p^c - p^a) < 0. \]
	Therefore, $d > 0$. 
	
	Since $0 \in \{b, d, f\}$ but $d, f > 0$, necessarily $b = 0$. But then, taking congruences modulo $q$ in \eqref{eq12}, we deduce that $p^a \equiv 0 \pmod q$, which is impossible since $p$ and $q$ are distinct primes. Finally, we can conclude that $n$ cannot be $\Delta$-primitive.
	
	If $n \in \mathbb{N}(\Delta)$ but is not $\Delta$-primitive, then we can write $n = \alpha^2 m$ for some $\alpha \geq 2$ and some $\Delta$-primitive $m$. Note that $m = p^{x'} q^{y'}$, where $x' \leq x$, $y' \leq y$ but $(x', y') \neq (x, y)$. Since $q > p^{x'}$ and $m$ is $\Delta$-primitive, it follows from the first part of the proof that $m \notin \mathbb{N}(\Delta)$, which is absurd.   
\end{proof}

\begin{corollary}
	\label{contrarr.p^xq^y.no.tiene.prop.Delta.q>p^x}
	Let $x, y \in \mathbb{N}$, with $y \geq 2$, and let $p, q$ be distinct primes. For each pair $(p, x)$, there are only finitely many primes $q$ such that $p^x q^y \in \mathbb{N}(\Delta)$.
\end{corollary} 

\begin{proof}
	If $p^x q^y \in \mathbb{N}(\Delta)$, then $q < p^x$, by  \Cref{p^xq^y.no.tiene.prop.Delta.q>p^x}. Thus, fixing $p$ and $x$, we have only finitely many options for $q$.
\end{proof}

As an application of  \Cref{contrarr.p^xq^y.no.tiene.prop.Delta.q>p^x}, consider a number $n$ of the form $9q^y$, with $y \geq 2$. If $n \in \mathbb{N}(\Delta)$, then it must be that $q < 9$, i.e., $q \in \{2, 5, 7\}$.

%Over the course of the next three lemmas, we prove that for every pair of primes $p, q$, no number of the form $p^2 q^2, p q^2$, or $p q$ satisfies the $\Delta$ condition.

%%%%%%%%%%%%%%%%%%%%%%%%%%%%%%%%%%%%%%%%%%%%%%%%%%%%%%%%%%%%%%%

\subsection{Integers of the form $p^k$, $pq$, $pq^2$ and $p^2q^2$}

\begin{theorem}
	\label{p^k.no.tiene.prop.Delta}
	If $p$ is prime and $k\in\mathbb{N}$, then $p^k \notin \mathbb{N}(\Delta)$. 
\end{theorem}

\begin{proof}
	First, observe that if $n = p^k$, then:
	\begin{equation*}
		D^*_n -\{0\} = \{p^{k - i} - p^i : 0 \leq 2i < k\}.
	\end{equation*}
	If $k\in\{1,2\}$, the result follows immediately from \Cref{p>2k-2_entonces.pk.no tiene.prop.Delta}. Thus, we can assume $k \geq 3$. We proceed by contradiction. Suppose $n$ is $\Delta$-primitive. By \eqref{ecuacion.Delta.prop}, we have 
	\begin{equation}
		p^{k - z} - p^z + p^{k - y} - p^y = p^{k - x} - p^x,
		\label{eq13}    
	\end{equation}
	for certain non-negative integers $x, y, z$ such that $2x, 2y, 2z < k$. Then, $k-x, k-y, k-z > 0$. If $0 \in \{x, y, z\}$, $n - 1$ would participate in \eqref{eq13}, contradicting  \Cref{cotasup.D^*capD^+}. Thus, $x, y, z > 0$, and we can divide both sides of \eqref{eq13} by $p$. But then $n = p^2 m$, where $2 \leq m = p^{k - 2} \in \mathbb{N}(\Delta)$ (recall that $k \geq 3$), contradicting the primitivity of $n$. Therefore, we conclude that $n$ cannot be $\Delta$-primitive.
	
	If $n \in \mathbb{N}(\Delta)$ but is not $\Delta$-primitive ($k \geq 3$), then $n = \alpha^2 m$, where $\alpha \geq 2$ and $m$ is $\Delta$-primitive. But $m = p^h \in \mathbb{N}(\Delta)$, for some $h \in [k - 2]$, which contradicts what was shown in the primitive-case.   
\end{proof}

\begin{lemma}
	\label{p^2q^2.no.tiene.prop.Delta}
	If $p$ and $q$ are primes, then $p^2q^2\notin\mathbb{N}(\Delta)$.
\end{lemma}

\begin{proof}
	The case $p=q$ follows from  \Cref{p^k.no.tiene.prop.Delta}. Assume $q<p$, and let $n=p^2q^2$. By  \Cref{p^xq^y.no.tiene.prop.Delta.q>p^x}, we may also assume $q<p<q^2<p^2$.
	
	Under these hypotheses, $\sqrt{n}=pq$, and since $p^2>pq$, the divisors of $n$ strictly smaller than $\sqrt{n}$ are exactly $\{1,q,p,q^2\}$. By  \Cref{Delta.prop.caracterizacion}, any $\Delta$-triple $(x,y,z)$ for $n$ has $\{x,y,z\}\subseteq\{q,p,q^2\}$, leaving the four candidates
	\[ (q,p,p),\quad (q,q^2,q^2),\quad (p,q^2,q^2),\quad (q,p,q^2). \]
	We discard each in turn, writing \eqref{ecuacion.Delta.prop} explicitly for each tuple.
	
	\smallskip
	$(q,p,p)$: yields $q(p^2-1)=2(pq^2-p)$. Reducing modulo $p$ gives $-q\equiv 0\pmod{p}$, forcing $p=q$, a contradiction.
	
	\smallskip
	$(q,q^2,q^2)$: yields $q(p^2-1)=2(p^2-q^2)$. Modulo $q$: $0\equiv 2p^2\pmod{q}$, so $q=2$ (as $\gcd(p,q)=1$). Then $2(p^2-1)=2(p^2-4)$, i.e.\ $-1=-4$, absurd.
	
	\smallskip
	$(p,q^2,q^2)$: yields $p(q^2-1)=2(p^2-q^2)$. Modulo $p$: $0\equiv -2q^2\pmod{p}$, so $p=2$, contradicting $q<p$.
	
	\smallskip
	$(q,p,q^2)$: yields $q(p^2-1)=p(q^2-1)+(p^2-q^2)$, which is equivalent to $(p-1)(p-q)(q-1)=0$. Since $p,q$ are distinct primes, all three factors are nonzero, a contradiction.
\end{proof}

\begin{lemma}
	If $p$ and $q$ are primes, then $p q^2 \notin \mathbb{N}(\Delta)$.
	\label{pq^2.no.tiene.prop.Delta}
\end{lemma}

\begin{proof}
	Let $n = p q^2$. If $p = q$, we use  \Cref{p^k.no.tiene.prop.Delta}. Then, suppose that $p\neq q$. If $p=2$, we apply \Cref{2impar.no.tiene.prop.Delta}. If $(p,q)=(3,2)$, then $n=12\notin\mathbb{N}(\Delta)$. If $q=2$ and $p\ge 5$, then \Cref{p>2k-2_entonces.pk.no tiene.prop.Delta} applies. Therefore, assume from now on that $p$ and $q$ are odd. 
	
	Since $D_n^* = \{n - 1, p q - q, |q^2 - p|\}$, we have 
	\[ D_n^+ = \{p q - q + |q^2 - p|\} \cup (D_n^* + (n - 1)) \cup 2 D_n^*. \]
	Obviously, $p q - q + |q^2 - p|\notin \{2(pq-q),2|q^2-p|\}$ and $2d\neq d$ for all $d\in D_n^*$. Moreover, 
	\[ \{p q - q + |q^2 - p|, 2(pq-q),2|q^2-p|\}\cap\{n-1\}=\varnothing, \]
	by  \Cref{cotasup.D^*capD^+}. If $2(pq-q)=|q^2-p|$ or $2|q^2-p|=pq-q$, then $2|q^2-p|\equiv 0\pmod{q}$, i.e., $2p\equiv 0\pmod{q}$: this is impossible because $p$ and $q$ are distinct odd primes. Hence, $D^*_n \cap D^+_n = \varnothing$. 
\end{proof}

\begin{lemma}
	\label{pq.no.tiene.prop.Delta}
	If $p$ and $q$ are primes, then $p q \notin \mathbb{N}(\Delta)$.
\end{lemma}

\begin{proof}
	If $p = q$, we use \Cref{p^k.no.tiene.prop.Delta}. If $p \neq q$, it immediately follows from \Cref{p>2k-2_entonces.pk.no tiene.prop.Delta}.
%	we can assume without loss of generality that $p < q$. Then, with $n = p q$, $D^*_n = \{n - 1, q - p\}$ and
%	\begin{equation*}
%		D^+_n = \{2(n - 1), n - 1 + q - p, 2(q - p)\}.    
%	\end{equation*}
%	At this point, using the usual arguments, it is quickly verified that $D^*_n \cap D^+_n = \varnothing$.
\end{proof}

\begin{theorem}
	\label{pq_pq^2_p^2q^2.no.tiene.prop.Delta}
	If $p$ and $q$ are prime numbers, then
	\[ \{p q, p^2q, pq^2, p^2 q^2\} \subset \mathbb{N} - \mathbb{N}(\Delta). \]
\end{theorem}

\begin{proof}
	This is the content of  \Cref{p^2q^2.no.tiene.prop.Delta,pq^2.no.tiene.prop.Delta,pq.no.tiene.prop.Delta}.
\end{proof}

%The bound $\tau(n)\geq 8$ admits a natural reformulation via differences of squares. For each divisor $d$ of $n$ with $d\leq\sqrt{n}$, the pair
%\[
%A_d=d+\tfrac{n}{d},\qquad B_d=\tfrac{n}{d}-d,
%\]
%satisfies $A_d^2-B_d^2=4n$ with $A_d\equiv B_d\pmod 2$, and conversely every such representation arises in this way. Writing $r_\ominus(N)$ for the number of representations of $N$ as $A^2-B^2$ with $A>B\geq 0$ and $A\equiv B\pmod 2$, one has the elementary identity
%\[
%r_\ominus(4n)\;=\;\left\lceil\frac{\tau(n)}{2}\right\rceil.
%\]
%Hence  \Cref{cor:tau-cota} can be paraphrased: every $n\in\mathbb{N}(\Delta)$ admits at least four representations of $4n$ as a difference of squares. The set $D^*_n$ coincides with $\{B_d:d\mid n,\,d\leq\sqrt{n}\}$, and the $\Delta$ condition translates into the existence of three such $B$'s satisfying $B_1=B_2+B_3$ with $B_1\neq n-1$. In this language, $\mathbb{N}(\Delta)$ is the set of integers $n$ for which the projection of the integer points of the hyperbola $uv=n$ onto the axis $v-u$ contains a non-trivial additive triple.

%%%%%%%%%%%%%%%%%%%%%%%%%%%%%%%%%%%%%%%%%%%%%%%%%%%%%%%%%%%%%%

\subsection{Integers of the form $p^kq$}

Having ruled the products of two primes with exponents $\le 2$, we now turn to the family $p^kq$ ($p,q$ primes, $k\in\mathbb{N}$), where allowing $k$ to grow yields not an obstruction but a complete classification. We start with a technical lemma we need later to prove that 24 and 40 are the only $\Delta$-primitives of the form $p^k q$.

\begin{lemma}
	\label{2^i-1q-2^j-1=2^i+j-q}
	Let $i, j \in \mathbb{N}$, and let $q$ be an odd number $\ge 3$. If
	\begin{equation}
		|2^{i - 1} q - 2^{j - 1}| = |2^{i + j} - q|,
		\label{eq9}
	\end{equation}
	then $(i, j, q) \in \{(1, 2, 5), (2, 1, 3)\}$.   
\end{lemma}

\begin{proof}
	If $q > 2^{i + j}$ in \eqref{eq9}, then $2^{i - 1} q - 2^{j - 1} > 2^{2i + j - 1} - 2^{j - 1} > 0$. Thus,
	\[ 0 < 2^{j - 1} (2^{i + 1} - 1) = q (1 - 2^{i - 1}) \leq 0, \]
	which is absurd. Therefore, it follows that $2^{i + j} > q$ in \eqref{eq9}. Successively, using similar arguments, it is easy to also eliminate the absolute value on the left-hand side of \eqref{eq9}. Finally, \eqref{eq9} is equivalent to
	\begin{equation}
		2^{i - 1} q - 2^{j - 1} = 2^{i + j} - q.
		\label{eq10}
	\end{equation}
	Since the right-hand side of \eqref{eq10} is odd, necessarily one of the two terms on the left-hand side must be odd. Therefore, it must be either $i = 1$ or $j = 1$. Substituting $i = 1$ into \eqref{eq10} and simplifying, we obtain $2q = 2^{j - 1} 5$, which implies $j = 2$ and $q = 5$. If instead we substitute $j = 1$ into \eqref{eq10}, we get $q - 1 = 2^{i - 1} (4 - q)$. Since $q - 1 > 0$, it must also be that $4 - q > 0$. Given that by hypothesis $q$ is an odd number $\geq 3$, the only option is $q = 3$, which in turn implies $i = 2$.   
\end{proof}

\begin{lemma}
	\label{p^kq.prim.24.40}
	Let $n = p^k q$, where $p$ and $q$ are distinct primes, $q$ is odd, and $k \geq 2$. If $n$ is $\Delta$-primitive, then $n \in \{24, 40\}$. 
\end{lemma}

\begin{proof}
	Since $n \in \mathbb{N}(\Delta)$, there exist non-negative integers $e_1, \ldots, e_6$ such that $e_1 + e_2 = e_3 + e_4 = e_5 + e_6 = k$ and
	\begin{equation}
		|p^{e_1} q - p^{e_2}| = |p^{e_3} q - p^{e_4}| + |p^{e_5} q - p^{e_6}|.
		\label{eq4}
	\end{equation}
	If $e_i > 0$ for all $i$, we can divide both sides of \eqref{eq4} by $p$, obtaining that $m = p^{k - 2} q \in \mathbb{N}(\Delta)$ (recall that $k \geq 2$). But then $n = p^2 m$, which contradicts the primitivity of $n$. Thus, $0 \in \{e_1, \ldots, e_6\}$. Note that necessarily $e_2, e_4, e_6 > 0$ in \eqref{eq4}, thanks to  \Cref{cotasup.D^*capD^+}. 
	
	If $e_1 = 0$, then \eqref{eq4} becomes
	\begin{equation}
		|q - p^k| = |p^{e_3} q - p^{e_4}| + |p^{e_5} q - p^{e_6}|. 
		\label{eq5}
	\end{equation}
	If $e_3 = 0$ or $e_5 = 0$, we find $|q - p^k|$ as a summand on the right-hand side of \eqref{eq5}. Both cases lead to clear absurdities. If $e_3, e_5 > 0$, then the right-hand side of \eqref{eq5} becomes a multiple of $p$: this represents another absurdity because the left-hand side of \eqref{eq5} is not. Therefore, we can conclude that $e_1 > 0$ in \eqref{eq4}, i.e., the left-hand side of \eqref{eq4} is a multiple of $p$, and moreover $0 \in \{e_3, e_5\}$. 
	
	Suppose now that $e_3 = 0$ and $e_5 > 0$. After reducing \eqref{eq4} modulo $p$, we get that $p\mid q$, which is obviously false. The same would happen if $e_3 > 0$ and $e_5 = 0$. Hence, the only possibility is that $e_3 = e_5 = 0$. After all these steps, we have converted \eqref{eq4} into
	\begin{equation}
		|p^{e_1} q - p^{e_2}| = 2 |q - p^{e_1 + e_2}|.
		\label{eq6}
	\end{equation}
	Reducing \eqref{eq6} modulo $p$, we obtain $2q \equiv 0 \pmod p$, which is true only if $p = 2$. We can then rewrite \eqref{eq6} as
	\begin{equation}
		|2^{e_1 - 1} q - 2^{e_2 - 1}| = |q - p^{e_1 + e_2}|.
		\label{eq7}
	\end{equation}
	We have shown that if $n$ is $\Delta$-primitive, then $n = 2^{e_1 + e_2} q$, where the integers $e_1, e_2$ and $q$ satisfy \eqref{eq7}. By  \Cref{2^i-1q-2^j-1=2^i+j-q}, the only triples $(e_1, e_2, q)$ of natural numbers satisfying \eqref{eq7} are $(1, 2, 5)$ and $(2, 1, 3)$. Therefore, we conclude that $n \in \{24, 40\}$. 
\end{proof}

\begin{theorem}
	\label{p^kq.tiene.prop.Delta.iff...}
	Let $n = p^k q$, where $p$ and $q$ are primes and $k \in \mathbb{N}$. Then, $n \in \mathbb{N}(\Delta)$ if and only if $n \in \{2^{2h + 1} q : q \in \{3, 5\}, h \in \mathbb{N}\}$. Moreover, $n$ is $\Delta$-primitive if and only if $n\in\{24,40\}$.
\end{theorem}

\begin{proof}
		If $k = 1$, then $n \notin \mathbb{N}(\Delta)$ by \Cref{pq.no.tiene.prop.Delta}. If $p = q$, then $n \notin \mathbb{N}(\Delta)$ by \Cref{p^k.no.tiene.prop.Delta}. If $q = 2$ and $p>2$, then $n \notin \mathbb{N}(\Delta)$ by \Cref{2impar.no.tiene.prop.Delta}. If $k=2$, then $n \notin \mathbb{N}(\Delta)$ by \Cref{pq_pq^2_p^2q^2.no.tiene.prop.Delta}. Therefore, if $n \in \mathbb{N}(\Delta)$, then $q$ has to be odd, $p\neq q$ and $k\ge 3$.
				
		Let $k \geq 3$. If $n$ is $\Delta$-primitive, we use \Cref{p^kq.prim.24.40}. If $n$ is not $\Delta$-primitive, then there exists an integer $e > 0$ such that $n = p^{2e}\, p^{k - 2e} q$ and $m = p^{k - 2e} q$ is $\Delta$-primitive. If $k - 2e \in \{0, 1\}$, then $m \in \{q, p q\}$, and neither $q$ nor $p q$ lies in $\mathbb{N}(\Delta)$, by \Cref{p^k.no.tiene.prop.Delta} and \Cref{pq.no.tiene.prop.Delta}; this contradicts that $m\in\mathbb{N}(\Delta)$. Hence, $k - 2e \geq 2$, so $m$ satisfies the hypotheses of \Cref{p^kq.prim.24.40}, which gives $m \in\{24, 40\}$, i.e., $p = 2$, $k - 2e = 3$ and $q \in \{3, 5\}$. Therefore, $n = 2^{2e + 3} q=2^{2h+1}q$, with $q \in \{3, 5\}$ and $h=e+1$.
		
		Conversely, since $24, 40$ are $\Delta$-primitive, every $2^{2h+1}q$ with $q \in \{3,5\}$ lies in $\mathbb{N}(\Delta)$ by \Cref{multiplos.cuadr.prop.delta}.
\end{proof}
%	If $k = 1$, we apply  \Cref{pq.no.tiene.prop.Delta}. If $k \geq 2$, suppose $p \neq q$, $q$ is odd, and $n \in \mathbb{N}(\Delta)$. If $n$ is $\Delta$-primitive, we use  \Cref{p^kq.prim.24.40}. If $n$ is not $\Delta$-primitive, then there exists an integer $e > 0$ such that $n = p^{2e} p^{k - 2e} q$ and $m = p^{k - 2e} q$ is $\Delta$-primitive. But again, by  \Cref{p^kq.prim.24.40}, it follows that $m \in \{24, 40\}$, i.e., $p = 2, k - 2e = 3$ and $q \in \{3, 5\}$. 

\begin{corollary}
	If $p$ and $q$ are distinct odd primes, consider a number $n$ of the form $p^2 q^3, p^3 q^3$, or $p^2 q^4$. If $n \in \mathbb{N}(\Delta)$, then $n$ is $\Delta$-primitive.
\end{corollary}

\begin{proof}
	Let $n \in \{p^2 q^3, p^3 q^3, p^2 q^4\}$. Suppose $n$ has the $\Delta$ property but is not $\Delta$-primitive. Then $n = \alpha^2 m$, where $\alpha \in \{p, q, p q\}$ and $m$ is $\Delta$-primitive. Thus,
	\[ m \in \{q, q^3, p^2 q, p q, p^3 q, p q^3, q^2, q^4, p^2 q^2\}. \]
	Thanks to \Cref{p^k.no.tiene.prop.Delta,pq_pq^2_p^2q^2.no.tiene.prop.Delta,p^kq.tiene.prop.Delta.iff...}, we know this is not possible.
\end{proof}

Based on experimental evidence, we conjecture that 24 and 40 are the only $\Delta$-primitive numbers using exactly two primes in their decomposition.

\begin{conjecture}
	Let $n$ be a natural number of the form $p^x q^y$, where $p$ and $q$ are distinct primes and $x, y\in\mathbb{N}$. If $n$ is $\Delta$-primitive, then $n \in \{24, 40\}$. 
\end{conjecture}

%%%%%%%%%%%%%%%%%%%%%%%%%%%%%%%%%%%%%%%%%%%%%%%%%%%%%%%%%%%%%%

\subsection{Products of three distinct primes}

 \Cref{pq.no.tiene.prop.Delta} cannot be extended to a product $p q r$ of three distinct primes: numbers like $105 = 3 \cdot 5 \cdot 7$ and $385 = 5 \cdot 7 \cdot 11$ would be counterexamples. However, we will analyze this new situation to understand in detail how the case $p q$ differs from the case $p q r$.

Let $p, q$, and $r$ be primes such that $p < q < r$. If $n = p q r$, then 
%\[ D_n = \{1, p, q, r, p q, p r, q r, n\}, \]
\[ D_n^* = \{n - 1, q r - p, p r - q, |p q - r|\}, \]
\[ D_n^+ = \{q r - p + p r - q, q r - p + |p q - r|, p r - q + |p q - r|\} \cup \]
\[ \cup (D_n^* + (n - 1)) \cup 2 D_n^*, \]
where the elements of $D_n^*$ are listed in decreasing order. This complete ordering is due to $1 < p < q < r, \ p q < p r < q r < n$. To determine whether $D_n^*$ and $D_n^+$ are disjoint, we examine all possible equalities $a + b = c$ with $a, b, c \in D_n^* - \{0\}$. Since $a + b > a,b$ (both summands being positive), neither summand can coincide with the sum. Moreover, by \Cref{cotasup.D^*capD^+}, $n - 1$ cannot belong to $D_n^* \cap D_n^+$. After these reductions, only the following cases remain:
%Imagine having to prove that $a + b \neq c$ for every triple of elements $a, b, c \in D_n^* - 0$. Clearly, $a + b \neq a$. Thanks to this observation and  \Cref{cotasup.D^*capD^+}, we can immediately discard many cases, leaving only the following:
\begin{center}
	\begin{multicols}{2}
		\begin{enumerate}
			\item $q r - p + p r - q = |p q - r|$,
			\item $q r - p + |p q - r| = p r - q$,
			\item $p r - q + |p q - r| = q r - p$,
			\item $2(q r - p) \in \{p r - q, |p q - r|\}$,
			\item $2(p r - q) = q r - p$,
			\item $2(p r - q) = |p q - r|$,
			\item $2|p q - r| = q r - p$,
			\item $2|p q - r| = p r - q$.
		\end{enumerate}
	\end{multicols}
\end{center}
%\begin{enumerate}
%	\item $q r - p + p r - q = |p q - r|$,
%	\item $q r - p + |p q - r| = p r - q$,
%	\item $p r - q + |p q - r| = q r - p$,
%	\item $2(q r - p) \in \{p r - q, |p q - r|\}$,
%	\item $2(p r - q) = q r - p$,
%	\item $2(p r - q) = |p q - r|$,
%	\item $2|p q - r| = q r - p$,
%	\item $2|p q - r| = p r - q$.
%\end{enumerate}
The impossibility of (1), (2), (4), and (6) is evident due to the complete ordering of the elements of $D_n^*$. 

If $r < p q$, then (7) can be rewritten as $p (2q + 1) = r (q + 2)$, so $q = k p - 2$, for some $k \in \mathbb{N}$. Then, $k (2p - r) = 3$, and thus $k \in \{1, 3\}$. If $k = 1$, then $q = p - 2 < p$. If $k = 3$, then $r = 2p - 1 < 3p - 2 = q$. Both cases contradict the hypotheses. If $p q < r$, then (7) is equivalent to $r (2 - q) = p (2q - 1)$, which is obviously false.

If $r < p q$, equality (8) can be rewritten as $q (2p + 1) = r (p + 2)$. Since $q (2p + 1)$ is odd, necessarily $p \neq 2$. Moreover, $q | (p + 2)$, which implies $q \leq p + 2$. But then $q = p + 2$, since $3 \leq p < p + 2 \leq q$, and thus $r = 2p + 1$. If $p q < r$, then we can convert (8) into $r (2 - p) = q (2p - 1)$, which is clearly absurd. 

We can rewrite (5) as $p (2r + 1) = q (r + 2)$. Then, $2r + 1 = k q$ and $r + 2 = h p$, for certain $k, h \in \mathbb{N}$. But then $p k q = q h p$, i.e., $k = h$, from which we deduce that $k (2p - q) = 3$. Hence, $k \in \{1, 3\}$. If $k = 1$, then $r = p - 2 < p$, a contradiction. If $k = 3$, then $r = 3p - 2$ and $q = 2p - 1$.

Regarding equality (3), note that it can be rewritten as
\[ (r + p + 1)(q - p - 1) = -p^2 - p - 1, \]
if $p q < r$, and
\[ (r - p + 1)(q - p + 1) = p^2 - p + 1, \]
if $r < p q$. The first case is immediately discarded. Regarding the second, we can say that $d_1 = q - p + 1$ and $d_2 = r - p + 1$ are complementary divisors of the integer $m = d_1 d_2 = p^2 - p + 1$. Therefore, by fixing $p$, we will have a finite number of pairs $(q, r)$ such that $p q r \in \mathbb{N}(\Delta)$. We also note that: 1) $1 < d_1 < \sqrt{m}$; 2) $m, d_1$ and $d_2$ are odd for $p > 2$. We can summarize all this analysis in the following proposition.

\begin{theorem}
	\label{pqr.prop.Delta}
	Let $p, q$, and $r$ be prime numbers such that $p < q < r$. Then $p q r \in \mathbb{N}(\Delta)$ if and only if one of these conditions is satisfied:
	\begin{enumerate}
		\item $(r - p + 1)(q - p + 1) = p^2 - p + 1$,
		\item $(q, r) \in \{(p + 2, 2p + 1), (2p - 1, 3p - 2)\}$.
	\end{enumerate}
	In particular, once $p$ is fixed, there are only finitely many pairs $(q, r)$ such that $p q r \in \mathbb{N}(\Delta)$. 
\end{theorem}

\begin{proof}
	It follows from the previous discussion.
\end{proof}

Notice that \Cref{pqr.prop.Delta}(1) is exactly \eqref{eq:xyz-identity} restricted to the prime setting; \Cref{prop:n=xyz} extends it to arbitrary $\Delta$-triples with $n = xyz$.

Let's see some examples of applying  \Cref{pqr.prop.Delta}. If $p = 2$, then $2 q r \in \mathbb{N}(\Delta)$ if and only if $(r - 1)(q - 1) = 3$ (1) or $(q, r) \in \{(4, 5), (3, 4)\}$ (2). Note that (1) implies $q - 1 = 1$, which contradicts $p < q$. Regarding (2), we see there are no pairs of primes. Therefore, we conclude that $2 q r \notin \mathbb{N}(\Delta)$. This is consistent with  \Cref{2impar.no.tiene.prop.Delta}, of which  \Cref{pqr.prop.Delta} is a particular case when $p = 2$. If $p = 3$, then $3 q r \in \mathbb{N}(\Delta)$ if and only if $(r - 2)(q - 2) = 7$ (1) or $(q, r) = (5, 7)$ (2). We see that (1) implies $q - 2 = 1$, which contradicts $p < q$. Thus, $3 q r \in \mathbb{N}(\Delta)$ if and only if $(q, r) = (5, 7)$, by (2). If $p = 5$, then $5 q r \in \mathbb{N}(\Delta)$ if and only if $(r - 4)(q - 4) = 21$ (1) or $(q, r) \in \{(7, 11), (9, 13)\}$ (2). Since $1 < q - 4 < \sqrt{21} < 5$, it follows from (1) that $q - 4 = 3$ and $r - 4 = 7$. Hence, $5 q r \in \mathbb{N}(\Delta)$ if and only if $(q, r) = (7, 11)$. Similarly, it is very easy to verify that $7 q r \in \mathbb{N}(\Delta)$ if and only if $(q, r) = (13, 19)$. We can summarize all this in the following corollary.

\begin{corollary}
	Let $p, q$, and $r$ be prime numbers such that $p < q < r$, and let $n = p q r$. If $p \leq 7$ and $n \in \mathbb{N}(\Delta)$, then $n \in \{105, 385, 1729\}$.
\end{corollary} 

\begin{proof}
	It follows from the previous discussion.  
\end{proof}

The next corollary generalizes some arguments used in the previous examples. 

\begin{corollary}
	\label{p^2-p+1.primo...}
	Let $p, q$, and $r$ be prime numbers such that $p < q < r$. If $p^2 - p + 1$ is prime and the set $\{p + 2, 2p \pm 1, 3p - 2\}$ contains at most one prime, then $p q r \notin \mathbb{N}(\Delta)$. 
\end{corollary}

\begin{proof}
	Condition (1) of  \Cref{pqr.prop.Delta} cannot be fulfilled acceptably if $p^2 - p + 1$ is prime, since this implies $q - p + 1 = 1$, contradicting the hypothesis that $p < q$.
	
	If the set $\{p + 2, 2p \pm 1, 3p - 2\}$ contains at most one prime, then it is impossible to form any pair of primes $(q, r)$ as required by condition (2) of  \Cref{pqr.prop.Delta}.
\end{proof}

The numbers 13, 67, and 79 are the smallest primes satisfying the hypotheses of  \Cref{p^2-p+1.primo...}.

Recall that $\tau:\mathbb{N}\to\mathbb{N}$ denotes the divisor counting function, $\tau(n)=|D_n|$, which is multiplicative and satisfies $\tau(p_1^{a_1}\cdots p_k^{a_k})=\prod_{i=1}^k(a_i+1)$, provided $p_i$ is prime for all $i\in[k]$. Then, the obstructions established along this section admit the following consequence for $\tau$.

\begin{corollary}
	\label{cor:tau-cota-refined}
	Let $n \in \mathbb{N}(\Delta)$ with $n \notin \{24, 40\}$ and $n$ not a 
	product of three distinct primes. Then,
	\[
	\tau(n) \in \{12\} \,\cup\, \{k \geq 15 : k \text{ composite}\}.
	\]
\end{corollary}

\begin{proof}
	Since $n$ admits a $\Delta$-triple $(x,y,z)$, we have that $\{1,x,y\}\subset[1,\sqrt{n})$ and $\{n/y,n/x,n\}\subset(\sqrt{n},n]$ are all distinct, giving $\tau(n)\geq 6$. If $\tau(n)$ is prime, then $n=p^{\tau(n)-1}$: excluded by \Cref{p^k.no.tiene.prop.Delta}.
	
	If $\tau(n)\in\{6,9\}$, then $n \in \{p^5, p^2 q, p q^2, p^8, p^2 q^2\}$: excluded by 
	\Cref{p^k.no.tiene.prop.Delta} and \Cref{pq_pq^2_p^2q^2.no.tiene.prop.Delta}.
	
	If $\tau(n) = 8$, then $n \in \{p^7,\ p^3 q,\ pqr\}$: excluded by 
	\Cref{p^k.no.tiene.prop.Delta,p^kq.tiene.prop.Delta.iff...}, and the hypothesis.
	
	If $\tau(n)\in\{10,14\}$, then $n \in \{p^9, p^4 q, p^{13}, p^6 q\}$: excluded by \Cref{p^k.no.tiene.prop.Delta,p^kq.tiene.prop.Delta.iff...}.
%	The values $\tau(n)\in\{6,7\}$ are excluded by  and \Cref{pq_pq^2_p^2q^2.no.tiene.prop.Delta}: every integer with $\tau(n)\in\{6,7\}$ is of the form $p^5$, $p^2q$, $pq^2$, or $p^6$ (with $p,q$ distinct primes), none of which lies in $\mathbb{N}(\Delta)$. We exclude the remaining 
%	forbidden values by listing the factorization shapes of $n$:
\end{proof}

%%%%%%%%%%%%%%%%%%%%%%%%%%%%%%%%%%%%%%%%%%%%%%%%%%%%%%%%%%%%%%

\subsection{Returning to split graphs}

Finally, we close this section with an important theorem about the $n$-simple induced cycles of $\Phi$, when $n \notin \mathbb{N}(\Delta)$ and is not a square.

%\begin{theorem}[\cite{vnsigma.factor.graph.paths}]
%	\label{ciclos.inducidos.en.Phi}
%	If $S$ is a split graph and $C$ is an induced cycle in $\Phi(S)$, then $|C|\leq 4$.
%\end{theorem}

\begin{theorem}
	\label{teo:return_to_Phi(S)}
	Let $S$ be a split graph, and let $C$ be an $n$-simple induced cycle in $\Phi(S)$. If $n$ is not a square and $n \notin \mathbb{N}(\Delta)$, then $|C| = 4$. 
\end{theorem}

\begin{proof}
	From \cite{vnsigma.factor.graph.paths}, we know that $|C| \in \{3, 4\}$. If $|C| = 3$, then $C$ would be of type 0 (see  \Cref{triangulos.permitidos}), by  \Cref{S.squarefree.implica.triang_tipo0}. Since $C$ is $n$-simple, it follows by  \Cref{triang.n-simple&tipo0.implica.Delta.prop} that $n \in \mathbb{N}(\Delta)$, which contradicts the hypothesis.
\end{proof}

It is natural to ask for an arithmetic invariant governing the induced $4$-cycles of $\Phi(S)$, playing the role that the $\Delta$ property plays for the type 0 triangles. The orientation type of a $4$-cycle analogous to $\Delta_0$ is the directed cycle on $a, b, c, d$ with arcs $ab$, $bc$, $cd$ and $ad$: a single arc $ad$ short-circuits the directed path $abcd$, so that the degree increase (in $S$) from $a$ to $d$ equals the sum of the three increases along the path. Determining the condition on $n$ that characterizes the realizability of such an $n$-simple $4$-cycle (the $C_4$ analogue of the $\Delta$ property) is, as far as we know, uncharted territory, and we leave it as a direction for future work.

%%%%%%%%%%%%%%%%%%%%%%%%%%%%%%%%%%%%%%%%%%%%%%%%%%%%%%%%

\section*{Acknowledgements}
This work was partially supported by Universidad Nacional de San Luis, grants PROICO 03-0723 and PROIPRO 03-2923, MATH AmSud, grant 22-MATH-02, Consejo Nacional de Investigaciones
Cient\'ificas y T\'ecnicas grant, PIP 11220220100068CO and Agencia I+D+I grants PICT 2020-00549 and PICT 2020-04064.

\bibliographystyle{abbrv}
\bibliography{citas__the_delta_property}	

@article{cheng2016split,
	title={Split graphs and Nordhaus--Gaddum graphs},
	author={Cheng, Christine and Collins, Karen L and Trenk, Ann N},
	journal={Discrete Mathematics},
	volume={339},
	number={9},
	pages={2345--2356},
	year={2016},
	publisher={Elsevier}
}

@article{jaume2025nullspace,
	title={On the nullspace of split graphs},
	author={Jaume, Daniel A and Schv{\"o}llner, Victor N and Panelo, Cristian and Pereyra, Kevin},
	journal={arXiv preprint arXiv:2512.00190},
	year={2025}
}

@article{vnsigma.2switch.degree,
	title={The 2-switch-degree of a graph},
	author={Schv{\"o}llner, Victor Nicolas and Pastine, Adri{\'a}n},
	journal={arXiv preprint arXiv:2511.23327},
	year={2025}
}

@article{vnsigma.factor.graph,
	title={Simple factor graphs associated with split graphs},
	author={Schv{\"o}llner, Victor Nicolas and Pastine, Adri{\'a}n},
	journal={arXiv preprint arXiv:2512.24252},
	year={2025}
}

@article{vnsigma.factor.graph.paths,
	title={Induced paths and cycles in factor graphs of split graphs},
	author={Schv{\"o}llner, Victor Nicolas and Pastine, Adri{\'a}n},
	journal={arXiv preprint arXiv:2603.14061},
	year={2026}
}

@article{siegel1929,
	title={\"Uber einige Anwendungen diophantischer Approximationen},
	author={Siegel, Carl Ludwig},
	journal={Abhandlungen der Preussischen Akademie der Wissenschaften. Physikalisch-mathematische Klasse},
	pages={41--69},
	year={1929}
}

@article{tyshkevich2000decomposition,
  title={Decomposition of graphical sequences and unigraphs},
  author={Tyshkevich, Regina},
  journal={Discrete Mathematics},
  volume={220},
  number={1-3},
  pages={201--238},
  year={2000},
  publisher={Elsevier}
}

@article{barrus.west.A4,
  title        = {The A\_4 structure of a graph},
  author       = {Barrus, Michael D. and West, Douglas B.},
  journal      = {Journal of Graph Theory},
  volume       = {69},
  number       = {2},
  pages        = {97--113},
  year         = {2012},
  publisher    = {Wiley},
  doi          = {10.1002/jgt.20639}
}

@article{vnsigma.2switch.unic.pseudof,
	title={2-switch: transition and stability on forests and pseudoforests},
	author={Schv{\"o}llner, Victor Nicolas and Pastine, Adri{\'a}n and Jaume, Daniel A},
	journal={arXiv preprint arXiv:2603.07439},
	year={2026}
}

@book{chartrand2010graphs,
  title={Graphs \& Digraphs},
  author={Chartrand, G. and Lesniak, L. and Zhang, P.},
  isbn={9781439895184},
  series={Textbooks in Mathematics},
  url={https://books.google.com.ar/books?id=rhrSBQAAQBAJ},
  year={2010},
  publisher={CRC Press}
}

@article{arikati1999realization,
  title={The realization graph of a degree sequence with majorization gap 1 is Hamiltonian},
  author={Arikati, Srinivasa R and Peled, Uri N},
  journal={Linear algebra and its applications},
  volume={290},
  number={1-3},
  pages={213--235},
  year={1999},
  publisher={Elsevier}
}
	
\end{document}